\documentclass[10pt]{article}
\usepackage{tikz}
\usepackage{eqnarray,algorithm}
\usepackage{subfigure}
\usepackage{amsmath,amsthm}
\usepackage{lipsum}
\usepackage{amsfonts}
\usepackage{graphicx}
\usepackage{epstopdf}
\usepackage{algorithmic}
\usepackage{authblk}
\usepackage[margin=1in]{geometry}
\ifpdf
  \DeclareGraphicsExtensions{.eps,.pdf,.png,.jpg}
\else
  \DeclareGraphicsExtensions{.eps}
\fi

\usepackage{enumitem}

\newtheorem{theorem}{\bf Theorem}[section]

\def \bes{\begin{eqnarray}}
\def \ees{\end{eqnarray}}
\def \bns{\begin{eqnarray*}}
\def \ens{\end{eqnarray*}}
\def \nnu{\nonumber}

\begin{document}

\title{A homotopy training algorithm for fully connected neural networks}

\author[$1$]{Qipin Chen}
\author[$1$]{Wenrui Hao}

\affil[$1$]{Department of Mathematics, Pennsylvania State University, University Park, PA 16802}




\maketitle

\begin{abstract}
	In this paper, we present a Homotopy Training Algorithm (HTA) to solve
	optimization problems arising from fully connected neural networks with complicated
	structures. The HTA dynamically
	builds the neural network starting from a simplified version and ending with the
	fully connected network via adding layers and nodes adaptively.
	Therefore, the corresponding optimization problem is easy to solve at
	the beginning and connects to the original model via a continuous path
	guided by the HTA, which provides a high probability of obtaining a global
	minimum. By gradually increasing the complexity of the model along the
	continuous path, the HTA provides a rather good solution to the original
	loss function. This is confirmed by various numerical results including
	VGG models on CIFAR-10. For example, on the VGG13 model with batch
	normalization, HTA reduces the error rate  by 11.86\% on test dataset compared with the traditional
	method.  Moreover, the HTA also allows us
	to find the optimal structure for a fully connected neural network by
	building the neutral network adaptively.
\end{abstract}

\section{Introduction}
	The deep neural network (DNN) model has been experiencing an
	extraordinary resurgence in many important artificial intelligence
	applications since the late 2000s. In particular, it has been able to
	produce state-of-the-art accuracy in computer vision
	\cite{simonyan2014very}, video analysis \cite{he2016deep}, natural
	language processing \cite{collobert2008unified}, and speech recognition
	\cite{sainath2013deep}. In the annual contest ImageNet Large Scale
	Visual Recognition Challenge (ILSVRC), the deep convolutional neural
	network (CNN) model has achieved the best classification accuracy since
	2012, and has exceeded human ability on such tasks since 2015
	\cite{markoff2015learning}. 
	%
	The success of learning through neural networks with large model size,
	i.e., deep learning, is widely believed to be the result of being able
	to adjust millions to hundreds of millions of parameters to achieve
	close approximations to the target function. The approximation is
	usually obtained by minimizing its output error over a training set
	consisting of a significantly large amount of samples.  Deep learning
	methods, as the rising star among all machine learning methods in
	recent years, have already had great success in many applications. Many
	advancements
	\cite{cang2018representability,cang2018integration,yin2018binaryrelax,yin2016quantization}
	in deep learning have been made in the last few years.  However, as the
	size of new state-of-the-art models continues to grow larger, they rely
	more heavily on efficient algorithms for training and making inferences
	from such models. This clearly places strong limitations on the
	application scenarios of DNN models for robotics \cite{konda2012real},
	auto-pilot automobiles \cite{hammerla2016deep}, and aerial systems
	\cite{maire2014convolutional}.   At present, there are two big challenges in
	fundamentally understanding deep neural networks: \begin{itemize} \item How to
	efficiently solve the highly nonlinear and non-convex
	optimization problems that arise during the training of a deep learning model.
	\item How to design a deep neural network structure for specific problems.
	\end{itemize}
	In order to solve these challenges, in this paper, we will present a new
	training algorithm based on the homotopy continuation method
	\cite{BHS,BHSW,MorganSommese1}, which has been successfully used to study
	nonlinear problems such as nonlinear differential equations
	\cite{HHHLSZ,HHHS,WHL}, hyperbolic conservation laws \cite{HHSSXZ,HY}, data
	driven optimization \cite{Hpar,HHarlim}, physical systems \cite{HNS1,HNS}, and
	some more complex free boundary problems arising from biology \cite{HCF,HF}. In
	order to tackle the nonlinear optimization problem in DNN, the homotopy
	training algorithm is designed and shows efficiency and feasibility for  fully
	connected neural networks with complex structures. The HTA also provides a new
	way to design a deep fully connected neural network with an optimal structure.
	This homotopy setup presented in this paper can  also be extended to other
	neural network such as CNN and RNN. In this paper, we will focus on fully
	connected DNNs only.
	The rest of this paper is organized as follows. We first introduce the HTA in
	Section 2 and then discuss the theoretical analysis in Section 3.  Several
	numerical examples are given in Section 4 to demonstrate the accuracy and
	efficiency of the HTA. Finally, applications of HTA to computer vision will be
	given in Section 5.

\section{Homotopy Training Algorithm}
\label{sec:main}
	The basic idea of HTA is to train a simple model at the
	beginning, then adaptively increase the structure's complexity, and
	eventually to train the original model. We will illustrate the idea of the
	homotopy setup by using a fully connected neural network with
	$\mathbf{x}=(x_1,\cdots x_n)^T$ as the input and
	$\mathbf{y}=(y_1,\cdots, y_m)^T$ as the output. More specifically, for a
	single hidden layer, the neural network (see Fig.  \ref{Fig:SL}) can be
	written as \bes \mathbf{y}=f(\mathbf{x})=W_2^T \sigma(W_1^T
	\mathbf{x}+\beta_1)+\beta_2,\label{DNN1}\ees where $\sigma$ is the
	activation function (for example, ReLU),  $W_1\in R^{n\times d_1}$ and
	$W_2\in R^{d_1\times m}$ are parameter matrices representing the
	weighted summation, $\beta_1 \in R^{d_1} $ and $\beta_2\in R^m$ are
	vectors representing bias, and $d_1$ is the number of nodes of the
	single hidden layer, namely, the width.
	Similarly, a fully connected neural network with two hidden layers (see
	Fig. \ref{Fig:2L})is written as \bes \mathbf{y}=f(\mathbf{x})=W_3^T
	\sigma(W_2^T \sigma(W_1^T
	\mathbf{x}+\beta_1)+\beta_2)+\beta_3,\label{DNN2}\ees where $W_1\in
	R^{n\times d_1}$, $W_2\in R^{d_1\times d_2}$, $W_3\in R^{d_2\times m}$,
	$\beta_1 \in R^{d_1} $, $\beta_2\in R^{d_2}$, $\beta_3\in R^{m}$ and $d_2$ is the width of the second layer.

	Then the homotopy continuation method is introduced to track the
	minimizer of (\ref{DNN1}) to the minimizer of (\ref{DNN2}) by setting
	\bes \mathbf{y}(t)&=&H_1(x;W_j,\beta_j,t)=(1-t)[W_2^T \sigma(W_1^T
	x+\beta_1)+\beta_2]\nnu\\&\ &+t[W_3^T \sigma(W_2^T
	\sigma(W_1^T x+\beta_1)+\beta_2)+\beta_3]~j=1,2,3.\label{Hom}\ees Then an optima of (\ref{DNN2})
	will be obtained by tracking the homotopy parameter $t$ from $0$ to
	$1$.
	The idea is that the model of (\ref{DNN1}) is easier to train than
	that of (\ref{DNN2}). Moreover, the homotopy setup will follow the universal
	approximation theory \cite{hornik1989multilayer,huang2006universal} to find an approximation trajectory to reveal the
	real nonlinear relationship between the input $x$ and the output $y$.
	Similarly, we can extend this homotopy idea to any two layers \bes
	H_i(x;\theta,t)=(1-t)y_i(x;\theta)+t y_{i+1}(x;\theta),\label{Hom_t}\ees where
	$y_{i}(x;\theta)$ is the approximation of a fully connected neural
	network with $i$ layers and $\theta$ represents parameters that are
	weights of the neural network. %
	In this case, we can train a fully connected neural network
	``node-by-node" and ``layer-by-layer." 	This computational algorithm
	can significantly reduce  computational costs of deep learning, which is
	used on large-scale data and complex problems.
	Using the homotopy setup, we are able to rewrite the ANN, CNN, and RNN
	in terms of a specific start system such as (\ref{DNN1}).  After designing
	a proper homotopy, we need to train this model with some data sets. In
	the homotopy setup, we need to solve the following optimization
	problem: \begin{equation} \theta(t)=\mathop{\arg\min}_{\theta}
	\sum_{j=1}^N \| H_i(X^j;\theta,t) - Y^j \|^2_{U}, \label{OPt}\end{equation} where
	$X^j$ and $Y^j$ represent data points  and $N$ is the number of data
	points in a mini-batch. In this optimization, the homotopy setup tracks
	the optima from a simpler optimization problem to a more complex one.
	The loss function in (\ref{OPt}) could  be changed to other types of
	entropy functions \cite{Goodfellow-et-al-2016,vapnik1999overview}.

	{\bf A simple illustration:} We consider a simple neural network with
	two hidden layers to approximate a scalar function $y=f(x)$ (the width
	of two hidden layers are 2 and 3 respectively). The
	detailed HTA algorithm for training this neural network is listed in
	{\bf Algorithm 1}. Then the neural network with a single layer in
	(\ref{DNN1}) gives us that $W_1\in R^{1\times 2}$, $W_2\in R^{2\times
	1}$, $\beta_1\in R^{2\times1}$, and $\beta_2\in R$. By denoting all the
	weights $W_1,~W_2,~\beta_1,\beta_2$ as $\theta_1$,  the optimization
	problem (\ref{OPt}) for $t=0$ is formulated as $min
	f_1(\theta_1)$. Then a minimizer $\theta_1^*=\{W_1^*,
	W_2^*,\beta_1^*\}$ satisfies the necessary condition
	\bes\nabla_{\theta_1} f_1(\theta_1^*)=0.\label{Sl1}\ees Similarly, for
	the neural network with two hidden layers, we have that, in
	(\ref{DNN2}), $W_2\in R^{2\times 3}$ and $\beta_2\in R^{3\times1}$ are
	changed, $W_3\in R^{3\times 1}$ and $\beta_3\in R$. Then all the new
	variables introduced by the second hidden layer ($W_3$, $\beta_3$ and
	part of $W_2$ and $\beta_2$) are denoted as $\theta_2$. The total
	variables $\theta=\theta_1\cup \theta_2$ formulate  the optimization
	problem of two hidden layers, namely, (\ref{OPt}) for
	$t=1$, as  $min f_2(\theta)$ and solve it by using (\ref{Hom_t}),
	which is equivalent to solving the following nonlinear equations: \bes
	Hom(\theta,t):=t\nabla_\theta
	f_2(\theta)+(1-t) \left( \begin{aligned}
		\nabla_{\theta_1} \tilde{f}_2(\theta)\\ \nabla_{\theta_2}
		\tilde{f}_2(\theta)-\nabla_{\theta_2} \tilde{f}_2(\theta^0)\\
	\end{aligned}\right)=0, \label{Hom1}\ees where $\tilde{f}_2$ is the
	objective function with  the activation function of the second hidden
	layer as the identity and $\theta^0$ is constructed as \[W_1^0=W_1^*,~
	W_2^0=[W_2^*,0,0],~ W_3^0=[1,0,0]^T \hbox{~and~}
	\beta_1^0=\beta_1^*,~\beta_2^0=[\beta_2^*,0,0]^T, ~\beta_3^0=0.\] By
	noticing that $\nabla_{\theta_1}
	\tilde{f}_2(\theta^0)=\nabla_{\theta_1} f_1(\theta_1^*)=0$, we have
	that $Hom(\theta^0,0)=0$, which implies that  the neural network with a
	single layer can be rewritten as a special form of the neural network
	with two layers. Then we can solve the optimization problem $min
	f_2(\theta)$ by tracking (\ref{Hom1}) with respect to $t$ from 0 to 1.
	This homotopy technique is quite often used in solving nonlinear
	equations \cite{HHHLSZ,HHHS,WHL}. However, in practice, we will use
	some advanced optimization methods for solving (\ref{OPt}), such as the
	stochastic gradient decent method, instead of solving nonlinear
	equations (\ref{Hom1}) directly.

\section{Theoretical analysis}
\label{sec:con}
	In this section, we analyze the convergence of HTA between any two
	layers, namely, from $i$-th layer to $i+1$-th layer. Assuming that we
	have a known minimizer of a fully connected neural network with $i$
	layers, we prove that we can get a minimizer by adding $i+1$-th layer
	through the HTA. First, we consider the expectation of the loss
	function as
	\begin{equation}
		\mathbb{E}_\xi(\mathcal{L}(H_{i}(x_\xi;\theta,t),y_\xi),
	\end{equation}
	where $\mathcal{L}$ is the categorical cross entropy
	loss function \cite{rao2019natural,bernico2018deep} and is defined as
	$\mathcal{L}(x,y) = -x(y) + log(\sum_j e^{x[j]})$, and $\xi$ is a
	random variable due to random algorithms for solving the optimization
	problem for each given $t$.
	For simplicity, we denote
	\begin{eqnarray}
		F(\theta;\xi,t)&:=&\mathcal{L}(H_{i}(x_\xi;\theta,t),y_\xi),\\
		f(\theta;t)&:=&\mathbb{E}_\xi[F(\theta;\xi,t)],\\
		G(\theta;\xi,t)&:=&\nabla_\theta F(\theta;\xi,t),
	\end{eqnarray} where the index $i$ does not contribute to the analysis
	and therefore is ignored in our notation.

	By denoting $\displaystyle\theta_*^t:=argmin_\theta f(\theta;t)$, we define our
	stochastic gradient scheme for any given $t$:
	\begin{equation}
		\label{sgd_scheme}
		\theta_{k+1} = \theta_k - \gamma_k G(\theta_k;\xi_k,t).
	\end{equation}
	First we have the following convergence theorem for any given $t$
	 with  the sigmoid activation function.
	\begin{theorem} {\bf (Nonconvex Convergence)} If $\nabla_\theta
		H_i(x;\theta,t)$ is bounded for a given $t$, namely, $
		|\nabla_\theta H_i(x;\theta,t)| \le M_t,$ and  $\{\theta_k\}$
		is contained in a bounded open set, supposing that
		(\ref{sgd_scheme}) is run with a step-size sequence satisfying
		\begin{equation}
			\sum_{k=1}^\infty \gamma_k = \infty \text{ and }
			\sum_{k=1}^\infty \gamma_k^2 < \infty,
		\end{equation} then we have 		\begin{equation}
			\mathbb{E}[\frac{1}{A_k}\sum_{k=1}^K \gamma_k\|\nabla
			f(\theta_k;t)\|_2^2]\rightarrow 0 \text{ as }
			K\rightarrow \infty \hbox{~ with~} A_k:=\sum_{k=1}^K \gamma_k.
		\end{equation}
	\label{nonconvex}

	\begin{proof}
		First, we prove the {\bf Lipschitz-continuous objective
		gradients} condition \cite{BCN}, which means that $f(\theta;t)$
		is $C^1$ and $\nabla f(\theta;t)$ is Lipschitz continuous
		with respect to $\theta$:
		\begin{itemize}
			\item {\bf  $f(\theta; t)$ is $C^1$.} Since $y_i(x_\xi;
				\theta)\in C^1$ for $\theta$, we have
					$H(x_\xi;\theta,t) \in C^1$. Moreover, since
					$\mathcal{L}(\cdot, y)$ is  $C^1$, we have that
					$F(\theta;\xi,t)\in C^1$ or that $\nabla_\theta
					F(\theta; \xi,t)$ is continuous. Considering
					\begin{equation} \nabla_\theta f(\theta;t) =
					\nabla_\theta \mathbb{E}_\xi(F(\theta;\xi,t)) =
					\mathbb{E}_\xi(\nabla_\theta F(\theta;\xi,t)),
					\end{equation} we have  that $\nabla_\theta
					f(\theta,t)$ is continuous  or that $f(\theta;
					t)\in C^1$.
			\item {\bf $\nabla f(\theta;t)$ is Lipschitz continuous.}
				Since	\begin{eqnarray} \nabla_\theta f(\theta;t)
				&=& \mathbb{E}_\xi[\nabla_x
				\mathcal{L}(H(x_\xi;\theta,t),y_\xi)\nabla_\theta
				H(x_\xi;\theta,t)], \end{eqnarray} we will prove that
					both
					$\nabla_x\mathcal{L}(H(x_\xi;\theta,t),y_\xi)$
					and $\nabla_\theta H(x_\xi; \theta,t)$ are
					bounded and Lipschitz continuous.  Because both
					$\sigma(x)=\frac{1}{1+e^{-x}}$ and
					$\sigma'(x)=\sigma(x)(1-\sigma(x))$ are
					Lipschitz continuous and  $\{\theta_k\}$ is
					bounded (assumption of Theorem
					\ref{nonconvex}), $\nabla_\theta H(x_\xi;
					\theta,t)$ is Lipschitz continuous. ($x_\xi$ is
					bounded because the size of our dataset is
					finite.)

				By differentiating  $\mathcal{L}(x,y)$, we have
					\begin{equation} \label{part1} \nabla_x
					\mathcal{L}(x,y) =
					(-\delta_y^1+\frac{e^{x_1}}{\sum_je^{x_j}},\cdots,-\delta_y^n+\frac{e^{x_n}}{\sum_je^{x_j}}),
					\end{equation} where  $\delta$ is the Kronecker
					delta. Since 	\begin{equation}
						\frac{\partial}{\partial
						x_k}\frac{e^{x_i}}{\sum_j e^{x_j}} =
						\begin{cases} &\frac{e^{x_i}\sum_j
							e^{x_j}-(e^{x_i})^2}{(\sum_j
							e^{x_j})^2} \quad k=i \\
							&\frac{-e^{x_i}
							e^{x_k}}{(\sum_j e^{x_j})^2}
							\quad k\ne i, \\ \end{cases}
					\end{equation} which implies that
					$\big|\frac{\partial}{\partial
					x_k}\frac{e^{x_i}}{\sum_j e^{x_j}}\big|\leq 2$,
					we see that $\nabla_x \mathcal{L}(\cdot, y)$ is
					Lipschitz continuous and bounded. Therefore,
					$\nabla_x\mathcal{L}(H(x_\xi;\theta,t),y_\xi)$
					is Lipschitz continuous and bounded. Thus,
					$\nabla f(\theta;t)$ is Lipschitz continuous.
		\end{itemize}

		Second, we prove the {\bf first and second moment limits} condition \cite{BCN}:
		\begin{itemize}
		\item[a.]  According to our theorem's assumption,
			$\{\theta_k\}$ is contained in an open set that is
				bounded. Since $f$ is continuous, $f$ is
				bounded;
		\item[b.] Since $G(\theta_k; \xi_k,t)=\nabla_\theta
			F(\theta_k;\xi_k,t)$ is continuous, we have
				\begin{equation}
					\mathbb{E}_{\xi_k}[G(\theta_k;\xi_k,t)]
					=
					\nabla_\theta\mathbb{E}_{\xi_k}[F(\theta_k;\xi_k,t)]
					= \nabla_\theta f(\theta_k;t).
				\end{equation} Therefore, \begin{equation}
					\nabla f(\theta_k;
					t)^T\mathbb{E}_{\xi_k}[G(\theta_k;
					\xi_k,t)]=\nabla f(\theta_k; t)^T\cdot
					\nabla f(\theta_k; t) = \|\nabla
					f(\theta_k; t)\|_2^2 \ge u\|\nabla
					f(\theta_k; t)\|_2^2 \end{equation} for
					$0<u\le 1$.
				
			On the other hand, we have \begin{equation}
			 \|\mathbb{E}_{\xi_k}[G(\theta_k;
			\xi_k,t)]\|_2=\|\nabla f(\theta_k; t)\|_2 \le
			u_G\|\nabla f(\theta_k; t)\|_2 \end{equation} for
			$u_G\ge 1$.  \item[c.] 		Since $\nabla F(\theta_k;
			\xi_k,t)$ is bounded for a given $t$, we have
				$\mathbb{E}_{\xi_k}[\|\nabla F(\theta_k;
				\xi_k,t)\|_2^2]$ is also bounded.  Thus,
				\begin{equation} \mathbb{V}_{\xi_k}[G(\theta_k;
					\xi_k,t)] :=
					\mathbb{E}_{\xi_k}[\|G(\theta_k;
					\xi_k,t)\|_2^2] -
				\|\mathbb{E}_{\xi_k}[G(\theta_k;
				\xi_k,t)]\|_2^2\le
				\mathbb{E}_{\xi_k}[\|G(\theta_k;
				\xi_k,t)\|_2^2], \end{equation} which implies
				that 		$\mathbb{V}_{\xi_k}[G(\theta_k;
				\xi_k,t)]$ is bounded.
		\end{itemize}

		We have checked assumptions 4.1 and 4.3 in  \cite{BCN}. By
		theorem 4.10 in  \cite{BCN}, with the diminishing step-size,
		namely, 		
		\begin{equation}
			\sum_{k=1}^\infty \gamma_k = \infty \text{ and }
			\sum_{k=1}^\infty \gamma_k^2 < \infty,
		\end{equation}
		the following convergence is obtained
		\begin{equation}
			\mathbb{E}[\frac{1}{A_k}\sum_{k=1}^K \gamma_k\|\nabla
			f(\theta_k;t)\|_2^2]\rightarrow 0 \text{ as }
			K\rightarrow \infty.
		\end{equation}
	\end{proof}

	\end{theorem}

	Second we theoretically explore the existence of solution path
	$\theta(t)$  when $t$ varies from $0$ to $1$ for the convex case.
	 The solution path of $\theta(t)$ might be complex for
	the non-convex case, i.e., bifurcations, and is hard to analyze
	theoretically. Therefore, we analyze the HTA theoretically on the
	convex case only but apply it to non-convex cases in the numerical
	experiments.
	We redefine our
	stochastic gradient scheme for the homotopy process as
	\begin{equation}
		\label{sgd_scheme}
		\theta_{k+1} = \theta_k - \gamma_k G(\theta_k;\xi_k,t_k),
	\end{equation}
	where $\gamma_k$ is the learning rate and $t_0=0$, $t_k\nearrow 1$.
	Instead of considering the local convergence of the HTA
	in a neighborhood of the global minimum, we proved the following
	theorem in a more general assumption, namely, $f$ is a convex and
	differentiable objective function with a bounded gradient.


	\begin{theorem} {\bf (Existence of solution path $\theta(t)$)}
		Assume that $f(\cdot,\cdot)$ is convex and differentiable and that
		$\|G(\theta;\xi,t)\| \le M$. Then for stochastic gradient scheme
		(\ref{sgd_scheme}), with a finite partition for $t$ between [0,1], we have 		
		\begin{equation}
			\lim_{n\rightarrow \infty} \mathbb{E}[f(\bar{\theta}_n,
			\bar{t}_n)] = f(\theta_*^1, 1),
		\end{equation}
		where $\bar{\theta}_n =
		\frac{\sum_{k=0}^n\gamma_k\theta_k}{\sum_{k=0}^n\gamma_k}$ and $\bar{t}_n =
		\frac{\sum_{k=0}^n\gamma_k t_k}{\sum_{k=0}^n\gamma_k}$.
	\end{theorem}

	\begin{proof}
		\begin{eqnarray}
		\mathbb{E}[\|\theta_{k+1}-\theta_*^{t_{k+1}}\|^2] &=&\mathbb{E}[\|\theta_k-\gamma_kG(\theta_k;\xi_k,t_k) -
			\theta_*^{t_k}\|^2] \nonumber\\ &\quad&
		- 2\mathbb{E}[\langle
			\theta_k - \gamma_kG(\theta_k;\xi_k,t_k) - \theta_*^{t_k},
			\theta_*^{t_{k+1}} - \theta_*^{t_k}\rangle]\nnu\\&\quad &
			+\mathbb{E}[\|\theta_*^{t_{k+1}} - \theta_*^{t_k}\|^2].
		\end{eqnarray}
		By defining
		\begin{equation}
			\begin{aligned}
				A_k=-2\mathbb{E}[\langle
				\theta_k-\gamma_kG(\theta_k;\xi_k,t_k) - \theta_*^{t_k} ,
				\theta_*^{t_{k+1}} - \theta_*^{t_k}\rangle] +
				\mathbb{E}[\|\theta_*^{t_{k+1}}-\theta_*^{t_k}\|^2],
			\end{aligned}
		\end{equation}
		we have $\sum_{k=0}^nA_k\le A<\infty$ since $t\in[0,1]$ has a finite partition. Therefore,  we obtain
		\begin{eqnarray}
		\mathbb{E}[\|\theta_{k+1} - \theta_*^{t_{k+1}}\|^2] \nonumber
			&=&\mathbb{E}[\|\theta_k - \gamma_kG(\theta_k;\xi_k,t_k) -
			\theta_*^{t_k}\|^2] + A_k \nonumber\\
			&=&\mathbb{E}[\|\theta_k-\theta_*^{t_k}\|^2] - 2\gamma_k \mathbb{E}
			[\langle G(\theta_k;\xi_k,t_k),\theta_k-\theta_*^{t_k}\rangle]
			+\gamma_k^2\mathbb{E}[\|G(\theta_k;\xi_k,t_k)\|^2] +
			A_k\nnu\\ &\le&\mathbb{E}[\|\theta_k-\theta_*^{t_k}\|^2] - 2\gamma_k
			\mathbb{E} [\langle G(\theta_k;\xi_k,t_k),\theta_k -
			\theta_*^{t_k}\rangle] +\gamma_k^2M^2 + A_k\nonumber.
		\end{eqnarray}
		Since
		\begin{eqnarray}
		\mathbb{E}[\langle
			G(\theta_k;\xi_k,t_k),\theta_k-\theta_*^{t_k}\rangle]
			&=&\mathbb{E}_{\xi_0,\cdots,\xi_{k-1}}[\mathbb{E}_{\xi_k}[\langle
			G(\theta_k;\xi_k,t_k) , \theta_k - \theta_*^{t_k}\rangle |\xi_0,\cdots,\xi_{k-1}]] \nonumber\\
			&=&\mathbb{E}_{\xi_0,\cdots,\xi_{k-1}}[\langle \nabla f(\theta_k;t_k) ,
			\theta_k - \theta_*^{t_k}\rangle|\xi_0,\cdots,\xi_{k-1}] \nonumber\\
			&=&\mathbb{E}[\langle \nabla f(\theta_k;t_k) , \theta_k -
			\theta_*^{t_k}\rangle],
		\end{eqnarray}
		we have
		\begin{eqnarray}
		\mathbb{E}[\|\theta_{k+1} - \theta_*^{t_{k+1}}\|^2] \le\mathbb{E}[\|\theta_k-\theta_*^{t_k}\|^2] -
			2\gamma_k\mathbb{E} [\langle \nabla f(\theta_k;t_k),\theta_k -
			\theta_*^{t_k}\rangle] +\gamma_k^2M^2 + A_k.
		\end{eqnarray}
		Due to the convexity of $f(\cdot, t_k)$, namely,
		\begin{equation}
			\langle \nabla f(\theta_k, t_k), \theta_k-\theta_*^{t_k}\rangle \ge
			f(\theta_k; t_k) - f(\theta_*^{t_k}; t_k),
		\end{equation} we conclude that
		\begin{eqnarray}
		\mathbb{E}[\|\theta_{k+1} - \theta_*^{t_{k+1}}\|^2] \le
			\mathbb{E}[\|\theta_k-\theta_*^{t_k}\|^2]- 2\gamma_k
			\mathbb{E}[ f(\theta_k;t_k) - f(\theta_*^{t_k},t_k)]
			+\gamma_k^2M^2 + A_k, \quad \label{eqn24}
		\end{eqnarray}
		or
		\begin{eqnarray}
		2\gamma_k\mathbb{E}[f(\theta_k;t_k) - f(\theta_*^{t_k};t_k)]\le
		 -\mathbb{E}[\|\theta_{k+1} - \theta_*^{t_{k+1}}\|^2 -
			\|\theta_k - \theta_*^{t_k}\|^2] + \gamma_k^2M^2 +
			A_k.\nonumber
		\end{eqnarray}
		By summing up $k$ from 0 to n,
		\begin{eqnarray}
		2\sum_{k=0}^n\gamma_k\mathbb{E}[f(\theta_k;t_k) -
			f(\theta_*^{t_k};t_k)]  &\le& -\mathbb{E}[\|\theta_{n+1} -
			\theta_*^{t_{n+1}}\|^2 - \|\theta_0 - \theta_*^{0}\|^2]  + M^2\sum_{k=0}^n\gamma_k^2+ \sum_{k=0}^n A_k \nonumber \\ &\le& D^2+
			M^2\sum_{k=0}^n\gamma_k^2 + \sum_{k=0}^n A_k,
		\end{eqnarray}
		where $D = \|\theta_0 - \theta_*^0\|$.
		Dividing $2\sum_{k=0}^n\gamma_k$ on both sides, we have
		\begin{eqnarray}
		\frac{1}{\sum_{k=0}^n\gamma_k} \sum_{k=0}^n\gamma_k
			\mathbb{E}[f(\theta_k;t_k) - f(\theta_*^{t_k};t_k)] \le
			\frac{D^2+ M^2\sum_{k=0}^n\gamma_k^2 + \sum_{k=0}^n
			A_k}{2\sum_{k=0}^n\gamma_k} \le \frac{D^2+
			M^2\sum_{k=0}^n\gamma_k^2 + A}{2\sum_{k=0}^n\gamma_k}.\nonumber
		\end{eqnarray}
		According to the convexity of $f(\cdot;\cdot)$ and Jensen's
		inequality \cite{jensen1906fonctions},
		\begin{equation}
			\frac{1}{\sum_{k=0}^n\gamma_k} \sum_{k=0}^n\gamma_k
			\mathbb{E}[f(\theta_k;t_k)] \ge \mathbb{E}[f(\bar{\theta}_n;\bar{t}_n)],
		\end{equation}
		where $\bar{\theta}_n =
		\frac{\sum_{k=0}^n\gamma_k\theta_k}{\sum_{k=0}^n\gamma_k}$ and
		$\bar{t}_n = \frac{\sum_{k=0}^n\gamma_k
		t_k}{\sum_{k=0}^n\gamma_k}$.

		Then we have
		\begin{eqnarray}
			\mathbb{E}[f(\bar{\theta}_n;\bar{t}_n)] -
			\frac{\sum_{k=0}^n\gamma_k f(\theta_*^{t_k};t_k)}{\sum_{k=0}^n\gamma_k}
			\le \frac{D^2+ M^2\sum_{k=0}^n\gamma_k^2 +
			A}{2\sum_{k=0}^n\gamma_k}.
		\end{eqnarray}

		We choose $\gamma_k$ such that $\sum_{k=0}^n\gamma_k = \infty$
		and $\sum_{k=0}^n\gamma_k^2 < \infty$, for example, $\gamma_k =
		\frac{1}{k}$. Taking $n$ to infinite, we have
		\begin{equation}
			 \lim_{n\rightarrow \infty}\mathbb{E}[f(\bar{\theta}_n,
			\bar{t}_n)] - \lim_{n\rightarrow \infty}
			\frac{\sum_{k=0}^n\gamma_k f(\theta_*^{t_k},
			t_k)}{\sum_{k=0}^n\gamma_k}\le 0.
		\end{equation}
		Since $t_k\nearrow 1$, $f(\cdot,\cdot)$ is continuous and
		$\displaystyle\sum_{k=0}^n\gamma_k = \infty$, then we have
		\begin{equation}
			\lim_{n\rightarrow \infty}\frac{\sum_{k=0}^n\gamma_k
			f(\theta_*^{t_k}, t_k)}{\sum_{k=0}^n\gamma_k} = f(\theta_*^1,
			1),
		\end{equation}
		which implies that
		\begin{equation}
			\lim_{n\rightarrow \infty} \mathbb{E}[f(\bar{\theta}_n,
			\bar{t}_n)] \leq f(\theta_*^1, 1).
		\end{equation}
Since $\bar{t}_n\rightarrow 1$, we have \begin{equation}
			\mathbb{E}[f(\lim_{n\rightarrow \infty}\bar{\theta}_n,
			1)] \leq f(\theta_*^1, 1).
		\end{equation}
On the other hand, 	$\theta_*^1$ is the global minimum due to the convexity of $f$, and we have	$\displaystyle\mathbb{E}[f(\lim_{n\rightarrow \infty}\bar{\theta}_n,
			1)] \geq f(\theta_*^1, 1).$ Thus, $\displaystyle\mathbb{E}[f(\lim_{n\rightarrow \infty}\bar{\theta}_n,
			1)] = f(\theta_*^1, 1)$ holds.

	\end{proof}

\section{Numerical Results}
In this section, we demonstrate the efficiency and the feasibility of the HTA by comparing it with the traditional method, the stochastic gradient descent method. For both methods, we used the same hyper parameters, such as learning rate (0.05), batch size (128), and the number of epochs (380) on the same neural network for various problems. Due to the non-convexity of objective functions, both methods may get stuck at local optimas. We also ran the training process 15 times with different random initial guesses for both methods and reported the best results for each method.
\label{sec:experiments}
	\subsection{Function Approximations}
		{\bf Example 1 (Single hidden layer):} The first example we
		considered is using a single-hidden-layer connected
		neural network to approximate function
		\begin{equation}
			f(x) = \sin(x_1 + x_2 + \cdots + x_n),
		\end{equation}
		where $x = (x_1,x_2,\cdots,x_n)^T\in R^n$. The width of the
		single hidden layer NN is 20 and the width of the hidden layer
		of initial state of HTA is set to be 10.  Then the homotopy
		setup is written as \[H(x;\theta,t)=(1-t)y_1(x;\theta) + ty_2(x;\theta)\]
		where $y_1$ and $y_2$ are the fully connected NNs with 10 and
		20 as their width of hidden layers respectively. In particular, we have
		\begin{equation}
			y_{1}(x;\theta) = W_{21} \cdot r(W_{11} \cdot x + b_{11}) +
			b_{21},
		\end{equation}
		\begin{equation}
			y_{2}(x;\theta) = W_{2} \cdot r(W_{1} \cdot x + b_{1}) +
			b_{21},
		\end{equation}
		where $x \in \mathbb{R}^{n\times 1}$, $W_{11} \in
		\mathbb{R}^{10\times n}$, $b_{11}\in \mathbb{R}^{10\times 1}$,
		$W_{21}\in \mathbb{R}^{1\times 10}$, $b_{21}\in
		\mathbb{R}^{1\times 1}$,
		$W_{1}=\left( \begin{aligned}
			W_{11}\\
			W_{12}\\
		\end{aligned}\right) \in
		\mathbb{R}^{20\times n}$, $b_{1}=\left( \begin{aligned}
			b_{11}\\
			b_{12}\\
		\end{aligned}\right)\in \mathbb{R}^{20\times 1}$, and
		$W_{2}=\left(
			W_{21},
			W_{22}
		\right)\in \mathbb{R}^{1\times 20}$. We use the ReLU function
		$r(x)=\max\{0,x\}$ as our activation function.  For $n\leq 3$,
		we used the uniform grid points, where the sample points are
		the Cartesian products of uniformly sampled points of each
		dimension.  Then the size of the training data set is $10^{2n}$.
		For $n\geq 4$, we employed the sparse grid
		\cite{garcke2006sparse,smolyak1963quadrature} with level 6 as
		sample points. For each $n$, $90\%$ of the data set is used for
		training while $10\%$  is used for testing.  The loss
		curves of one-dimensional and two-dimensional cases are shown in Figs. \ref{sin_loss_1d}
		and \ref{sin_loss_2d}. By choosing $\Delta
		t=0.5$, the testing loss of HTA (for $t=1$) is lower than that
		of the traditional training algorithm.  Fig.
		\ref{sin_testing_plot} shows the comparison between the
		traditional method and HTA for for the one-dimensional case while Fig.
		\ref{sin_testing_plot_2d} shows the comparison of the two-dimensional case by
		using contour curves. All the results of up to $n=5$ are
		summarized in Table \ref{tab_testing_loss}, which lists the test
		loss between the HTA and the traditional training algorithm. It
		shows clearly that the HTA method is more efficient than the
		traditional method.

		{\bf Example 2 (multiple hidden layers):} The second example is
		using a two-hidden-layer fully connected neural network to
		approximate the same function in Example 1 for the multi-dimensional case. Since the approximation of the neutral
		network with a single hidden layer is not effective for $n>3$ (see
		Table \ref{tab_testing_loss}), we use a two-hidden-layer fully
		connected neural network with 20 nodes for each layer.  Then we
		use the following homotopy setup to increase the width of each
		layer from 10 to 20:
		\begin{equation}
			H_1(x;\theta,t) = (1-t)y_1(x;\theta) + ty_2(x;\theta),\quad			H_2(x;\theta,t) = (1-t)y_2(x;\theta) + ty_3(x;\theta),
		\end{equation}
		where $y_1$, $y_2$,  and $y_3$ represent neural networks with
		width (10,10), (10,20), and (20,20) respectively. The rationale is
		that the first homotopy function, $H_1(x;\theta,t)$, increases
		the width of the first layer while the second homotopy
		function, $H_2(x;\theta,t)$, increases the width of the second
		layer.
		The size of the training data and the strategy of choosing $\Delta
		t$ is the same as in Example 1.  Table \ref{tab_testing_loss_hd}
		shows the results of the approximation, and Fig.
		\ref{sin_testing_loss_hd} shows the testing curves  for $n=5$ and $n=6$. The HTA achieves higher accuracy than the traditional
		method.

	\subsection{Parameter Estimation}
		Parameter estimation often requires a tremendous number of model
		evaluations to obtain the solution information on the parameter
		space \cite{CMV15,Hpar,Goodman}. However, this large number of
		model evaluations becomes very difficult and even impossible for
		large-scale computational models \cite{hma:05,MS94}. Then a
		surrogate model needs to be built in order to approximate the
		parameter space. Neural networks provide  an effective way to build
		the surrogate model. But an efficient training algorithm of
		neural networks is needed to obtain an effective approximation
		especially for limited sample data on parameter space. We will
		use the Van der Pol equation as an example to illustrate the
		efficiency of HTA on the parameter estimation.

{\bf Example 3:}	We applied the HTA to estimate the
		parameters of the Van der Pol equation:
		\begin{equation}
			y'' - \mu(k - y^2)y' + y = 0 \hbox{~with~} y(0)=2 \hbox{~and ~} y'(0)=0.
		\end{equation}
 In order to estimate the parameters $\mu$ and $k$ for a given data $\tilde{y}(t;\mu,k)$, we first use single-hidden-layer fully connected neural network to build a surrogate model with  $\mu$ and $k$ as inputs and $y(1)$ as the output. Our training data set
	is chosen on $1\le \mu \le 10$ and $1\le k \le 10$ with 8,281 (with $0.1$ as the mesh size).  This neural network is trained by both the traditional method and the HTA. The testing dataset is 961 uniform grid points on $11\le \mu_i \le 14$ and $11\le k_i \le 14$ with $0.1$ as the mesh size for both $\mu$ and $k$
 The testing loss curves of the traditional method and HTA are shown in Fig. \ref{vdp}: after $5\times10^4$ steps, the testing loss is $0.007$ for HTA and $0.220$ for the traditional method. We also compared these two surrogate models (traditional method and HTA) with the numerical ODE solution $y(1;\mu,k)$ in Fig. \ref{vdp_equation}. This comparison shows that the approximation of the HTA is closer to the ODE model than the traditional method.

Once we built surrogate models, then we moved to a parameter estimation step for any given data $\tilde{y}(t;\mu,k)$ to solve the following optimization problem
		\begin{equation}
\min_{\mu,k} (S(\mu,k)-\tilde{y}(1))^2,\label{PE}
		\end{equation}
where $S(\mu,k)$ is the surrogate neural network model and $\tilde{y}(1)$ is the data when $t=1$. In our example, we generated ``artificial data" on the testing dataset. We use the SDG to solve the optimization problem with $\rho=k=11$ as the initial guess. We define the error of the parameter estimation below:
		\begin{equation}
Err_{PE}=\frac{\sum_{i=1}^n \sqrt{(\mu^*_i-\mu_i)^2+(k^*_i-k_i)^2}}{n},
		\end{equation}
where $(\mu_i,k_i)$ is the sample point and $(\mu^*_i,k^*_i)$ is the optima of (\ref{PE}) for a given $\tilde{y}(1)$. Then the error of HTA is 0.71 while the error of the traditional method is 1.48. We also list some results of parameter estimation for different surrogate models in Table \ref{para}. The surrogate model created by the HTA provides smaller errors than the traditional method for parameter estimation.

\section{Applications to Computer Vision}
	Computer vision is one of the most common applications in the field of
	machine learning \cite{krizhevsky2012imagenet,rowley1998neural}. It has
	diverse applications, from designing navigation systems for
	self-driving cars \cite{bojarski2016end} to counting the number of
	people in a crowd \cite{chan2008privacy}. There are many different
	models that can be used for detection and classification of objects.
	Since our algorithm focuses on the fully connected neural networks, we
	will only apply our algorithm to computer vision models with fully
	connected neural  networks. Therefore, in this section, we will use
	different Visual Geometry Group (VGG) models \cite{simonyan2014very} as
	an example to illustrate the application of HTA to computer vision. In
	computer vision, the VGG models use convolutional layers to extract
	features of the input picture, and then flatten the output tensor to be
	a 512-dimension-long vector. The output long vector will be sent into
	the fully connected network (See Fig. \ref{VGG} for more details). In
	order to demonstrate the efficiency of HTA, we will apply it to the
	fully connected network part of the VGG models.  \subsection{Three
	States of HTA} The last stage of VGG models is a fully connected neural
	network that links convolutional layers of VGG to the classification
	categories of the Canadian Institute for Advanced Research (CIFAR-10)
	\cite{krizhevsky2009learning}. Then the input is the long vector
	generated by convolutional layers (the width is 512), while the output
	is the 10 classification categories of CIFAR-10. In order to train this
	fully connected neural network, we construct  a three states setup of
	HTA, which is shown in  Fig. \ref{3states}. In this section, we use $x$
	to represent the inputs generated by the convolutional layers ($x\in
	R^{512}$), and $\theta$ to represent the parameters for each state. The
	size of $\theta$ may change for different states.
	\begin{itemize}
		\item {\bf State 1:} For the first state, we construct a fully
			connected network with 2 hidden layers. The width of the $i$-th hidden
			layer is set to be $w_i$. Then it can be written as
			\begin{equation}
				y_{1}(x;\theta) = W_{31} \cdot r(W_{21} \cdot r(W_{11} \cdot x + b_{11}) +
				b_{21}) + b_{3},
			\end{equation}
			where $x \in \mathbb{R}^{512\times 1}$, $W_{11} \in \mathbb{R}^{w_1\times
			512}$, $b_{11}\in \mathbb{R}^{w_1\times 1}$, $W_{21}\in \mathbb{R}^{w_2\times
			w_1}$, $b_{21}\in \mathbb{R}^{w_2\times 1}$, $W_{31}\in \mathbb{R}^{10\times
			w_2}$, and $b_{3}\in \mathbb{R}^{10\times 1}$.  We use the ReLU function
			$r(x)=\max\{0,x\}$ as our activation function.

		\item {\bf State 2:} For the second state, we add $(512-w_1)$ nodes to the first hidden
				layer to recover the first hidden layer of the original model.
				Therefore, the formula of the second state becomes
				\begin{equation}
					y_{2}(x;\theta) = W_{31} \cdot r(\tilde{W}_{2} \cdot r(W_{1} \cdot
					x + b_{1}) + b_{21}) + b_{3},
				\end{equation}
				where $x \in \mathbb{R}^{512\times 1}$,
				$W_1 =\left( \begin{aligned}
					W_{11}\\
					W_{12}\\
				\end{aligned}\right) \in \mathbb{R}^{512\times 512}$,
				$b_1 =\left( \begin{aligned}
					b_{11}\\
					b_{12}\\
				\end{aligned}\right) \in \mathbb{R}^{512\times 1}$, and
				$\tilde{W}_2 =\left(
					W_{21},
					W_{22}
				\right) \in \mathbb{R}^{w_2\times 512}$. In particular, if we choose $W_{12}=0$ and $W_{22}=0$, we will recover $y_1(x;\theta)$ of state 1.

		\item {\bf State 3:} Finally, we recover the original structure of the VGG by adding $(512-w_2)$ nodes to the second hidden layer:
			\begin{equation}
				y_{3}(x; \theta) = W_{3} \cdot r(W_{2} \cdot r(W_{1} \cdot x +
				b_{1}) + b_{2}) + b_{3},
			\end{equation}

			where
			$W_3 =\left( \begin{aligned}
				W_{31},
				W_{32}
			\end{aligned}\right) \in \mathbb{R}^{10\times 512}$,
			$W_2 =\left( \begin{aligned}
				\tilde{W}_{2}\\
				\tilde{\tilde{W}}_{2}\\
			\end{aligned}\right) \in \mathbb{R}^{512\times 512}$,
			$b_2 =\left( \begin{aligned}
				b_{21}\\
				b_{22}\\
			\end{aligned}\right) \in \mathbb{R}^{512\times 1}$. $y_3(x;\theta)$ will be reduced to $y_2(x;\theta)$ if $\tilde{\tilde{W}}_{2}=0$ and $W_{32}=0$.
	\end{itemize}
	\subsection{Homotopic Path}
		In order to connect these three states, we use two homotopic paths
		thats are defined by the following homotopy functions: \begin{equation}
		H_{i}(x; \theta,t) = (1-t)y_{i}(x; \theta) + ty_{i+1}(x;
		\theta),~i=1,2.  \end{equation} In this homotopy setup, when $t=0$,  we
		already have an optimal solution $\theta_i$ for $i$-th state and want to
		find an optimal solution for $\theta$ for $i+1$-th state when $t=1$. By
		tracking $t$ from 0 to 1, we can discover a solution path $\theta(t)$
		since $y_{i}(x; \theta)$ is a special form of $y_{i+1}(x;
		\theta)$. Then the loss functions for the homotopy setup becomes \begin{equation}
		L_{i}(x, y; \theta, t) = L(H_i(x; \theta, t), y), \end{equation} where
		$L(x, y)$ is the loss function.
	\subsection{Training Process}
		We first optimize $L(y_{1}(x; \theta_1), y)$ for the first
		state. The model structure is relatively simple to solve, and
		it efficiently obtains a local minimum or even a global minimum
		for the loss function.
		Then the second setup is to optimize $L_{1}(x, y; \theta, t)$
		by using $\theta_1$ as an initial condition for $t=0$. By
		gradually tracking parameter $t$ to $1$, we obtain an optimal
		solution $\theta_2$ of $L(y_{2}(x; \theta), y)$. The
		third setup is to optimize $L_{2}(x, y; \theta, t)$ by tracking
		$t$ from $0$ ($\theta_2$) to $1$. Then we obtain an optimal
		solution, $\theta_3$, of $L(y_{3}(x;
		\theta), y)$.  Due to the continuous paths, the
		optimal solution $\theta_{i+1}$ of $i+1$-th state is connected to
		$\theta_i$ of the $i$-th state by the parameter $t$. In this way, we
		can build our complex network adaptively.
	\subsection{Numerical Results on CIFAR-10}
		We tested the HTA with the three-state setup on the CIFAR-10
		dataset. We used VGG11, VGG13, VGG16, and VGG19 with batch
		normalization \cite{simonyan2014very} as our base models.  Fig.
		\ref{vgg13_bn} shows the comparison of validation loss between
		HTA and the traditional method on the VGG13 model. 
		Using the HTA with VGG13 has a lower error rate (5.14\%) than the traditional
		method (5.82\%), showing that the HTA with VGG13 improves the error rate by
		11.68\%. All of the results for the different models are shown in Table
		\ref{vgg_result}. It is clearly seen that the HTA is more accurate than the
		traditional method for all of the different models. For example, the HTA with
		VGG11 results in an error rate of 7.02\% while the traditional method results
		in 7.83\% (an improvement of 10.34\%). In addition, the HTA with VGG19 has an
		error rate of 5.88\% compared to 6.35\% with the traditional model (an
		improvement of 7.40\%), and the HTA with VGG16 has an error rate of 5.71\%
		while the traditional method has an error rate of 6.14\% (a 7.00\%
		improvement).

	\subsection{The Optimal Structure of a Fully Connected Neural Network}
		\label{sec:alg}
		Since the HTA builds the fully connected neural network
		adaptively, it also provides a way for us to find the optimal
		structure of the fully connected neural network; for example,
		we can find the number of layers and the width for each layer.
		We designed an algorithm to find the optimal structure based on
		HTA. First, we began with a  minimal model; for example, in the
		VGG models of CIFAR-10, the minimal width of two hidden layers
		is 10 because of the 10 classification. Then we applied the HTA
		to the first hidden layer by adding ``node-by-node," and we
		optimized the loss function dynamically with respect to the
		homotopy parameter. If the optimal width of the first hidden
		layer was found, then the weights of new added nodes were close
		to zero after optimization. Then we moved to the second hidden
		layer and implemented the same process to train
		``layer-by-layer." When the weights of the new added nodes for
		the second hidden layer became close to zero, we terminated the
		process. 

		{\bf Numerical results on CIFAR-10:} When we applied the
		algorithm for finding the optimal structure to each VGG base
		model, we found that the results were more accurate than when
		we used the base model only. For VGG11 with batch
		normalization, we set $\delta t=1/2$ and $n_{epoch}=50$ and
		found the optimal structure whose widths of the first and
		second hidden layers are 480 and 20, respectively. The error
		rate with VGG11 was reduced to 7.37\% while the error rate of
		the base model was 7.83\%. In this way, our algorithm can reach
		higher accuracy but with a simpler structure. The rest of our
		experimental results for different VGG models are listed in
		Table \ref{BSF}.

\section{Conclusion}
	In this paper, we developed a homotopy training algorithm for the fully
	connected neural network models. This algorithm starts from a simple
	neural network and adaptively grows into a fully connected neural
	network with a complex structure. Then the complex neural network can
	be trained by the HTA to attain a higher accuracy. The convergence of
	the HTA for each $t$ is proved for the non-convex optimization that
	arises from fully connected neural networks  with a
	$C^1$ activation function. Then the existence of solution path
	$\theta(t)$ is demonstrated theoretically for the convex case although
	it exists numerically in the non-convex case. Several numerical
	examples have been used to demonstrate the efficiency and feasibility
	of HTA. We also proved the convergence of HTA to the local optima if
	the optimization problem is convex. The application of HTA to computer
	vision, using the fully connected part of VGG models on CIFAR-10,
	provides better accuracy than the traditional method. Moreover, the HTA
	method provides an alternative way to find the optimal structure to
	reduce the complexity of a neural network. In this paper, we developed
	the HTA for fully connected neural networks only, but we vision it as
	the first step in the  development of HTA for general neural networks.
	In the future, we will design a new way to apply it to more complex
	neural networks such as CNN and RNN so that the HTA can speed up the
	training process more efficiently. Since the
	structures of the CNNs and RNNs are very different from fully connected
	neural networks, we need to redesign the homotopy objective function in
	order to incorporate their structures, for instance, by including the
	dropout technique.

\section{Figures \& Tables}

The output for figure is:

	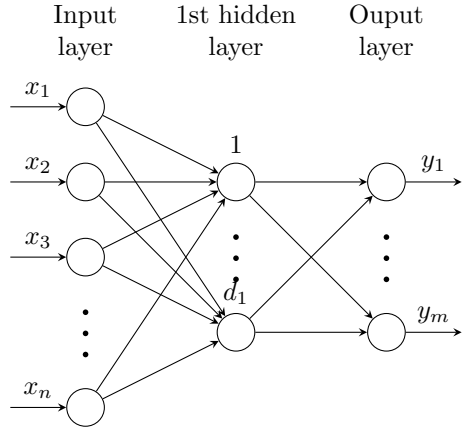
\begin{figure}
		\begin{center}
		\tikzset{%
		  every neuron/.style={
			circle,
			draw,
			minimum size=0.5cm
		  },
		  neuron missing/.style={
			draw=none,
			scale=2,
			text height=0.333cm,
			execute at begin node=\color{black}$\vdots$
		  },
		}
		\begin{tikzpicture}[x=1cm, y=1cm, >=stealth]

		\foreach \m/\l [count=\y] in {1,2,3,missing,4}
		  \node [every neuron/.try, neuron \m/.try] (input-\m) at (0,2.5-\y) {};

		\foreach \m [count=\y] in {1,missing,2}
		  \node [every neuron/.try, neuron \m/.try ] (hidden1-\m) at (2,1.5-\y) {};

		\foreach \m [count=\y] in {1,missing,2}
		  \node [every neuron/.try, neuron \m/.try ] (output-\m) at (4,1.5-\y) {};

			\foreach \l [count=\i] in {1,2,3,n}
		  \draw [<-] (input-\i) -- ++(-1,0)
			node [above, midway] {$x_{\l}$};

		\foreach \l [count=\i] in {1,d_1}
			\node [above] at (hidden1-\i.north) {$\l$};

		\foreach \l [count=\i] in {1,m}
		  \draw [->] (output-\i) -- ++(1,0)
			node [above, midway] {$y_{\l}$};

		\foreach \i in {1,...,4}
		  \foreach \j in {1,...,2}
			\draw [->] (input-\i) -- (hidden1-\j);

		\foreach \i in {1,...,2}
		  \foreach \j in {1,...,2}
			\draw [->] (hidden1-\i) -- (output-\j);

		\foreach \l [count=\x from 0] in {Input, 1st hidden, Ouput}
		  \node [align=center, above] at (\x*2,2) {\l \\ layer};
		\end{tikzpicture}
		\end{center}
		\caption{The structure of a neural network with a single hidden layer.}\label{Fig:SL}
	\end{figure}

	\begin{figure}
		\begin{center}
		\tikzset{%
		  every neuron/.style={
			circle,
			draw,
			minimum size=0.5cm
		  },
		  neuron missing/.style={
			draw=none,
			scale=2,
			text height=0.333cm,
			execute at begin node=\color{black}$\vdots$
		  },
		}

		\begin{tikzpicture}[x=1cm, y=1cm, >=stealth]

		\foreach \m/\l [count=\y] in {1,2,3,missing,4}
		  \node [every neuron/.try, neuron \m/.try] (input-\m) at (0,2.5-\y) {};

		\foreach \m [count=\y] in {1,missing,2}
		  \node [every neuron/.try, neuron \m/.try ] (hidden1-\m) at (2,1.5-\y) {};

		\foreach \m [count=\y] in {1,missing,2}
		  \node [every neuron/.try, neuron \m/.try ] (hidden2-\m) at (4,1.5-\y) {};

		\foreach \m [count=\y] in {1,missing,2}
		  \node [every neuron/.try, neuron \m/.try ] (output-\m) at (6,1.5-\y) {};

			\foreach \l [count=\i] in {1,2,3,n}
		  \draw [<-] (input-\i) -- ++(-1,0)
			node [above, midway] {$x_{\l}$};

		\foreach \l [count=\i] in {1,d_1}
			\node [above] at (hidden1-\i.north) {$\l$};

		\foreach \l [count=\i] in {1,d_2}
			\node [above] at (hidden2-\i.north) {$\l$};

		\foreach \l [count=\i] in {1,m}
		  \draw [->] (output-\i) -- ++(1,0)
			node [above, midway] {$y_{\l}$};

		\foreach \i in {1,...,4}
		  \foreach \j in {1,...,2}
			\draw [->] (input-\i) -- (hidden1-\j);

		\foreach \i in {1,...,2}
		  \foreach \j in {1,...,2}
			\draw [->] (hidden1-\i) -- (hidden2-\j);

		\foreach \i in {1,...,2}
		  \foreach \j in {1,...,2}
			\draw [->] (hidden2-\i) -- (output-\j);

		\foreach \l [count=\x from 0] in {Input, 1st hidden, 2nd hidden, Ouput}
		  \node [align=center, above] at (\x*2,2) {\l \\ layer};
		\end{tikzpicture}
		\end{center}
		\caption{The structure of a neural network with two hidden layers.}\label{Fig:2L}
	\end{figure}
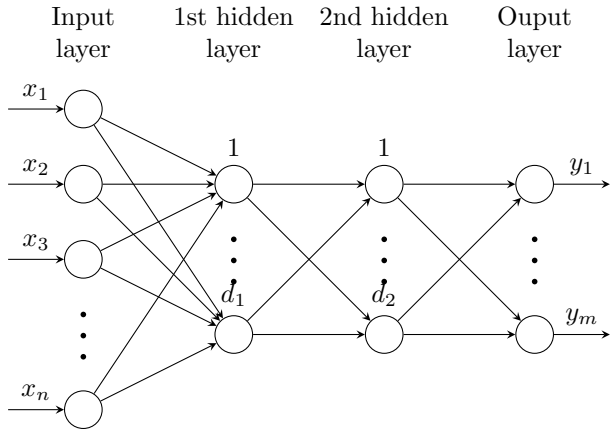

		\begin{figure*}[!t]
			\centering
			\subfigure[Training loss]
			{\includegraphics[width=3.5in, angle=0]{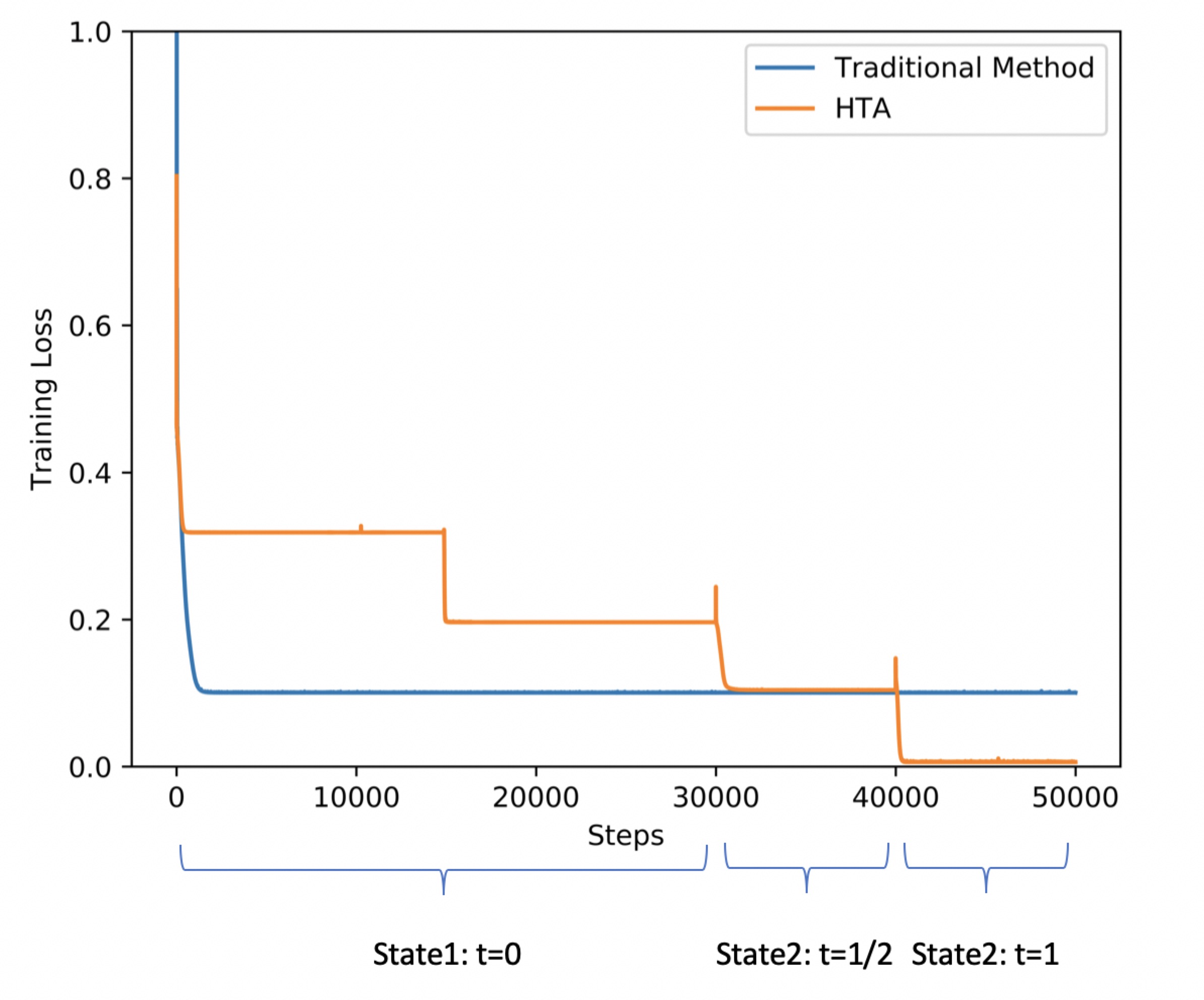}}
			\subfigure[Testing loss]
			{\includegraphics[width=3.5in, angle=0]{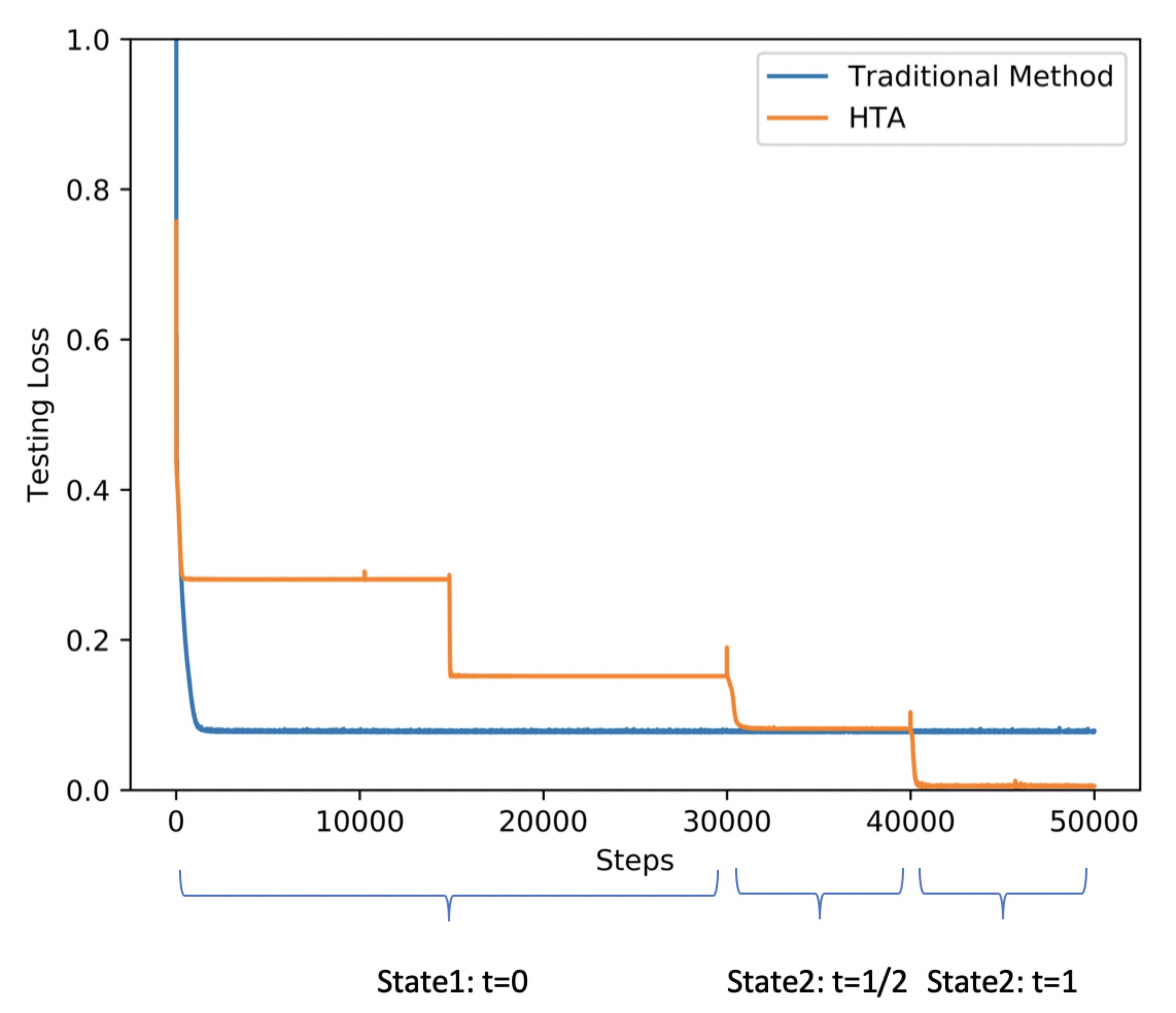}}
			\caption{Training/testing loss of approximating $\sin(x)$.}
			\label{sin_loss_1d}
		\end{figure*}
		\begin{figure*}[!t]
			\centering
			\subfigure[Training loss]
			{\includegraphics[width=3.5in, angle=0]{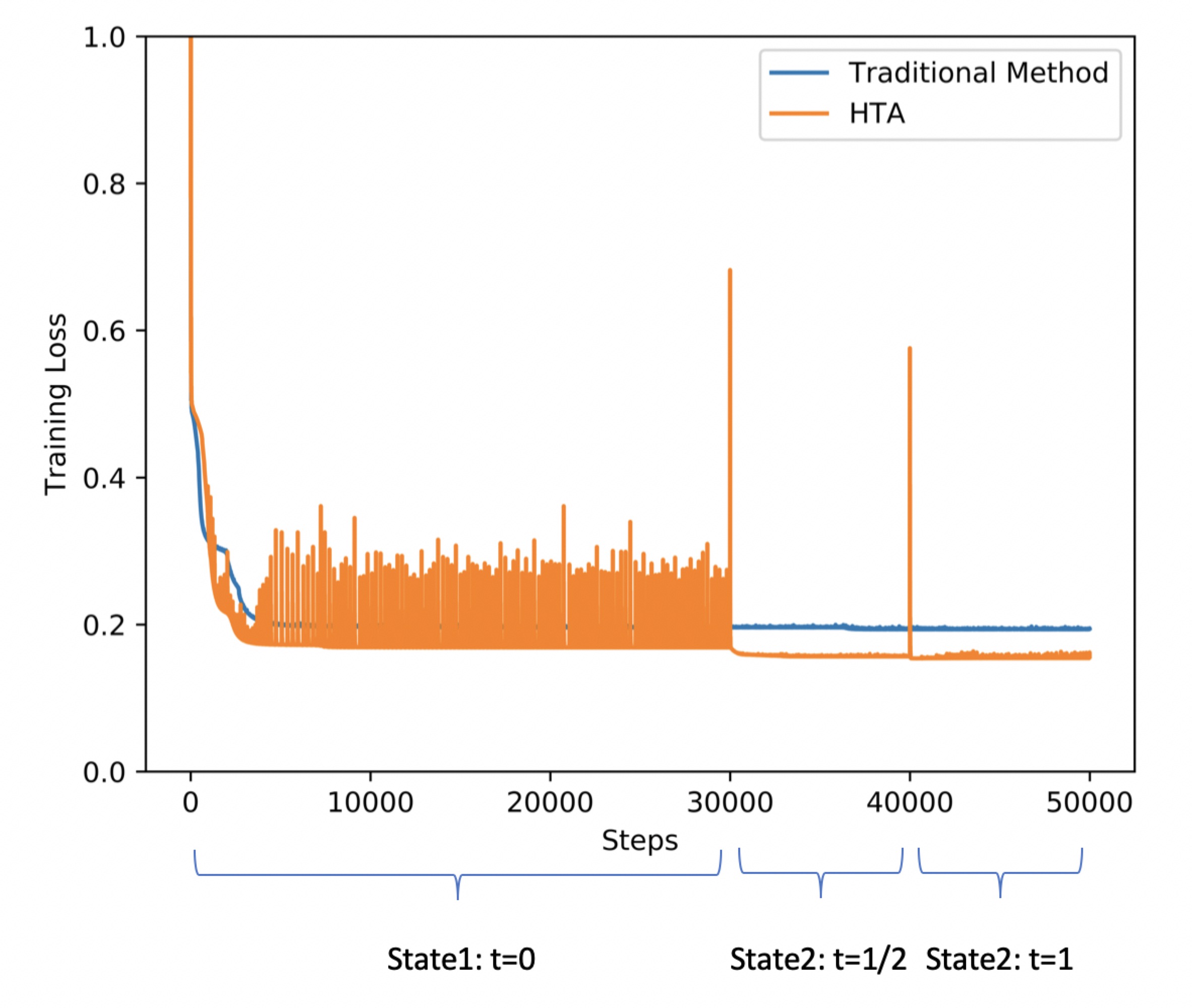}}
			\subfigure[Testing loss]
			{\includegraphics[width=3.5in, angle=0]{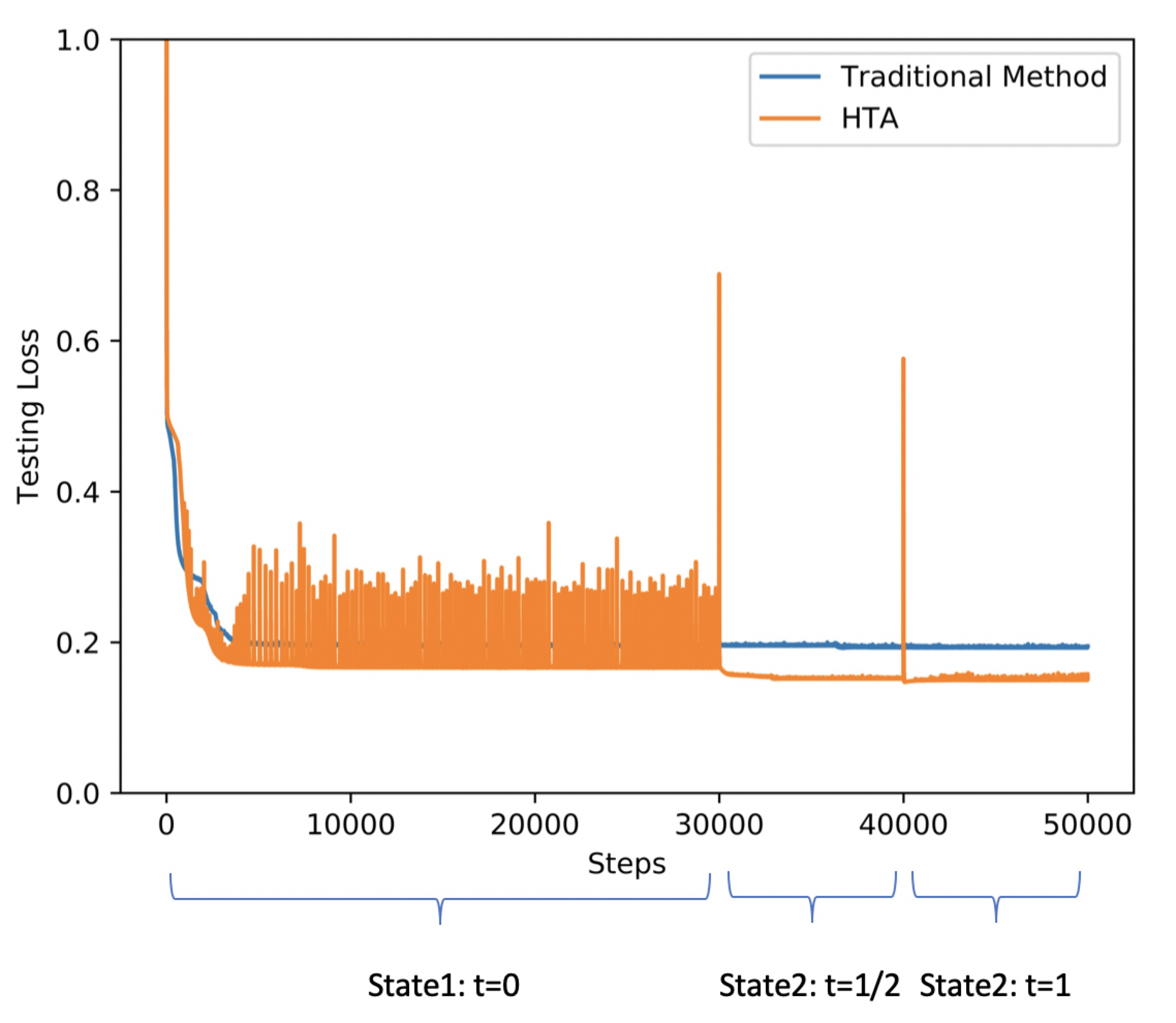}}
			\caption{Training/testing loss of approximating  $\sin(x_1 + x_2)$.}
			\label{sin_loss_2d}
		\end{figure*}

		\begin{figure*}[!t]
			\centering
			\subfigure[Approximation results of the traditional method]
			{\includegraphics[width=3.5in, angle=0]{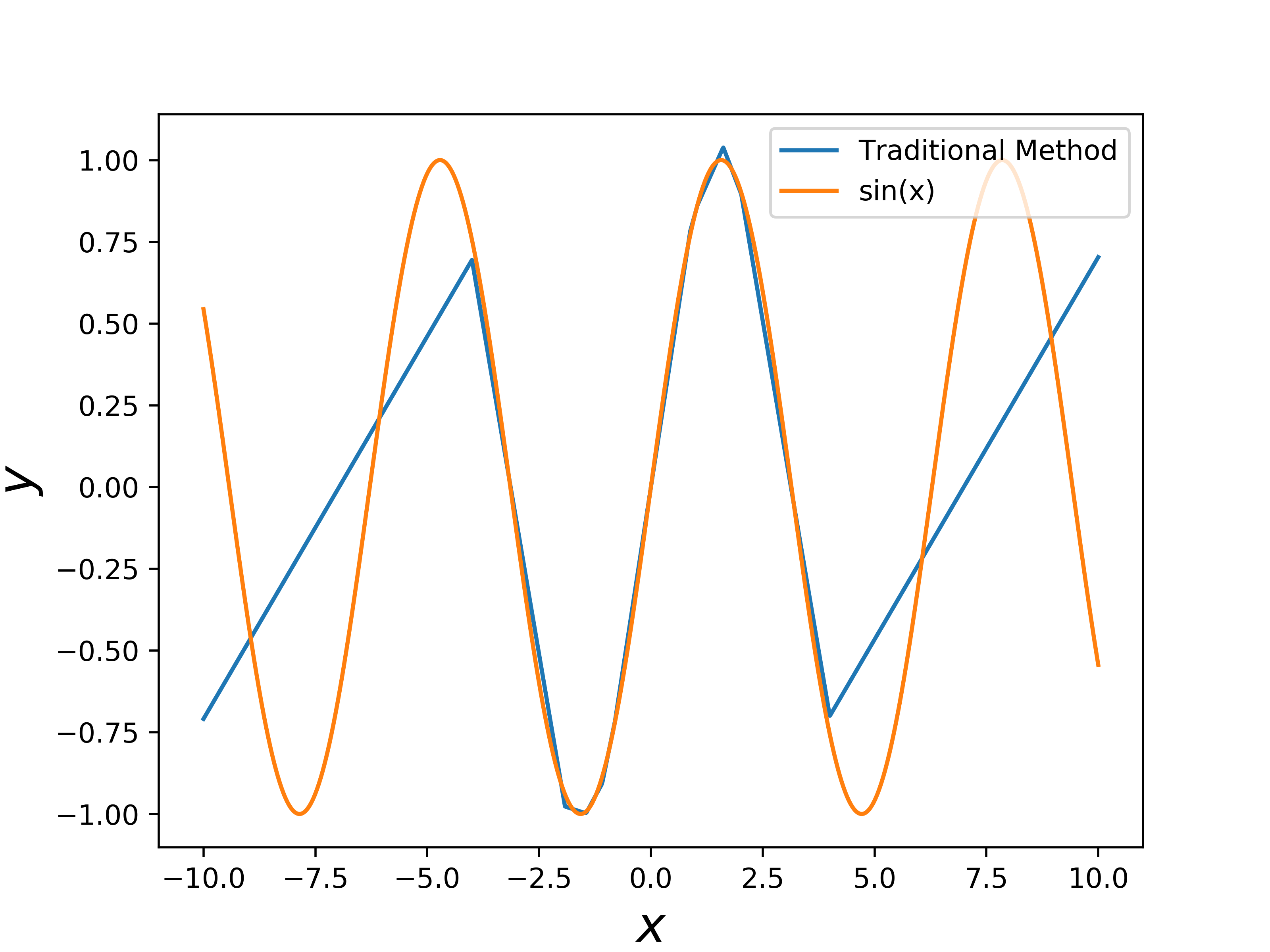}}
			\subfigure[Approximation results of the HTA]
			{\includegraphics[width=3.5in, angle=0]{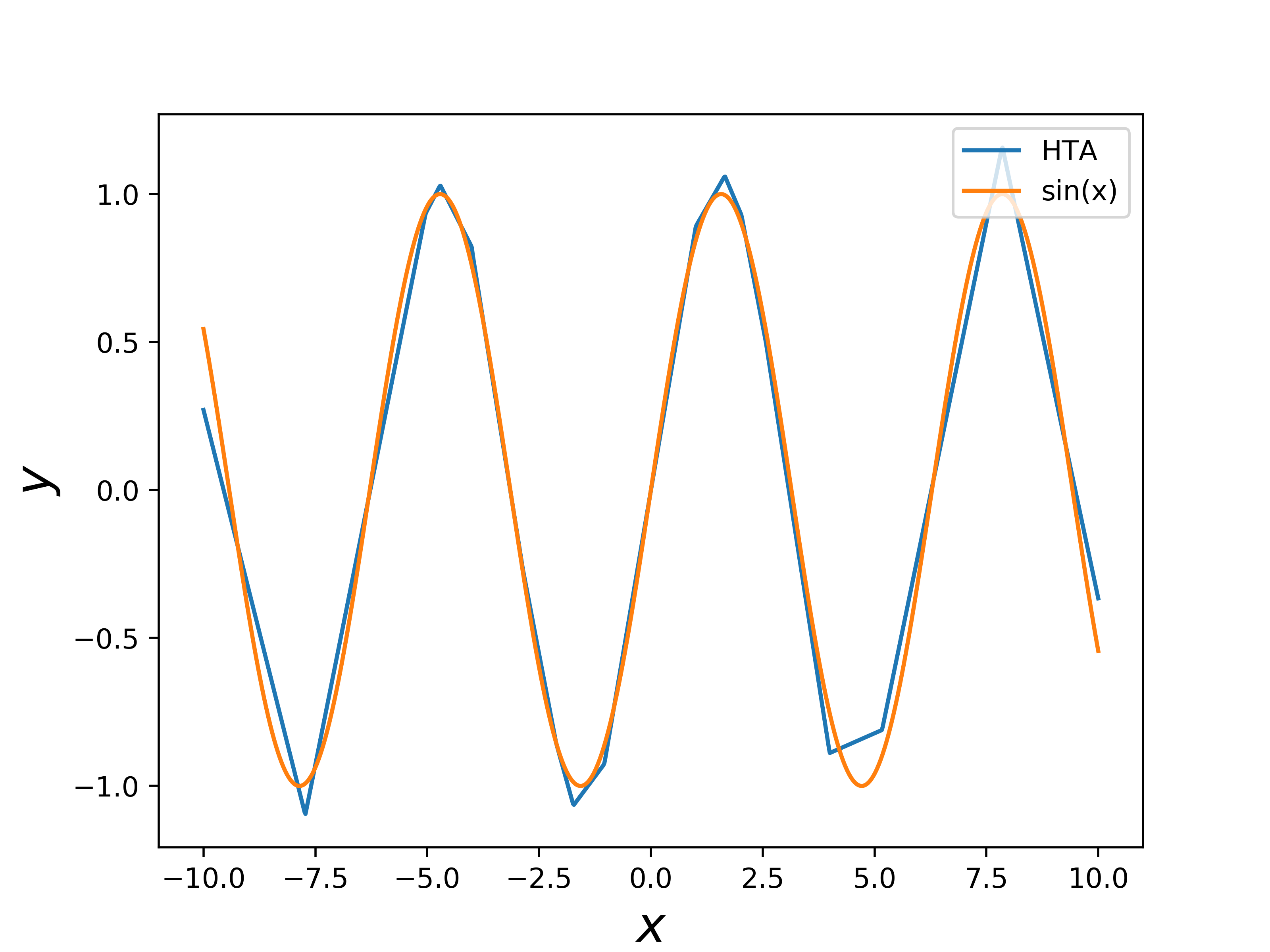}}
			\caption{Approximation results of $\sin(x)$.}
			\label{sin_testing_plot}
		\end{figure*}

		\begin{figure*}[!t]
			\centering
			\subfigure[Approximation results of the traditional method]
			{\includegraphics[width=3.5in, angle=0]{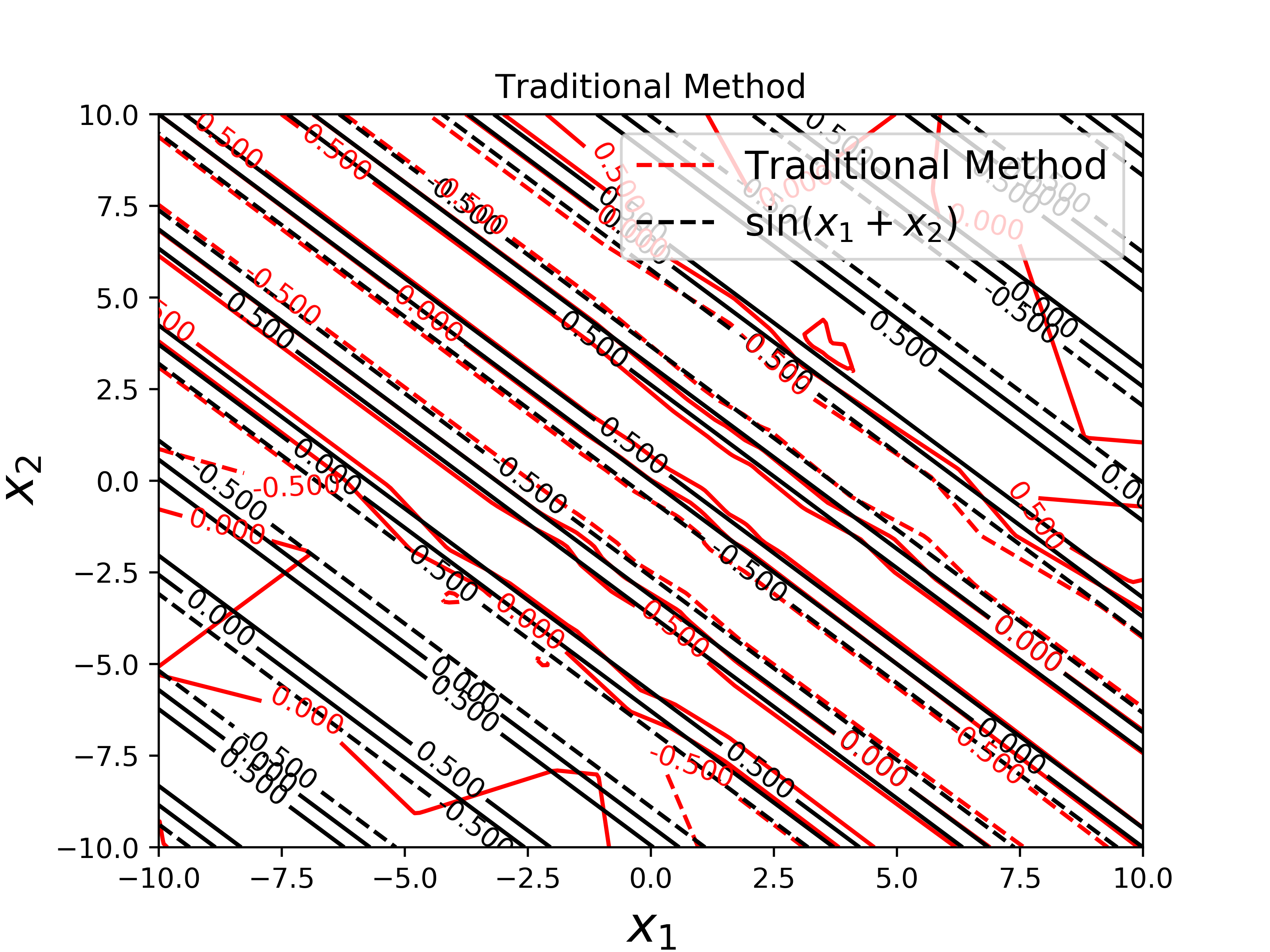}}
			\subfigure[Approximation results of the HTA]
			{\includegraphics[width=3.5in, angle=0]{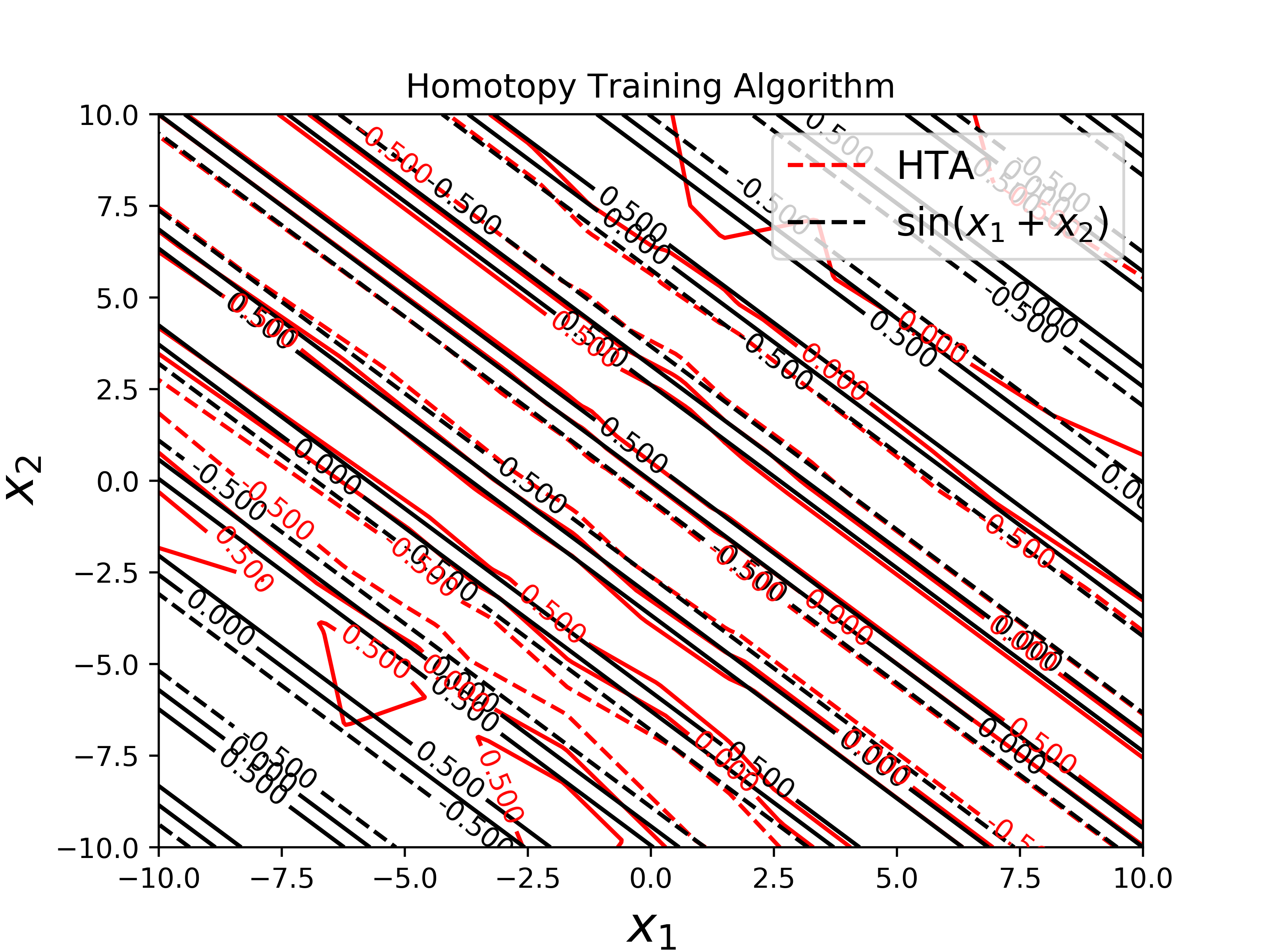}}
			\caption{Approximation results of $\sin(x_1 + x_2)$.}
			\label{sin_testing_plot_2d}
		\end{figure*}

		\begin{figure*}[!t]
			\centering
			\subfigure[Testing loss of $n=5$]
			{\includegraphics[width=3.5in, angle=0]{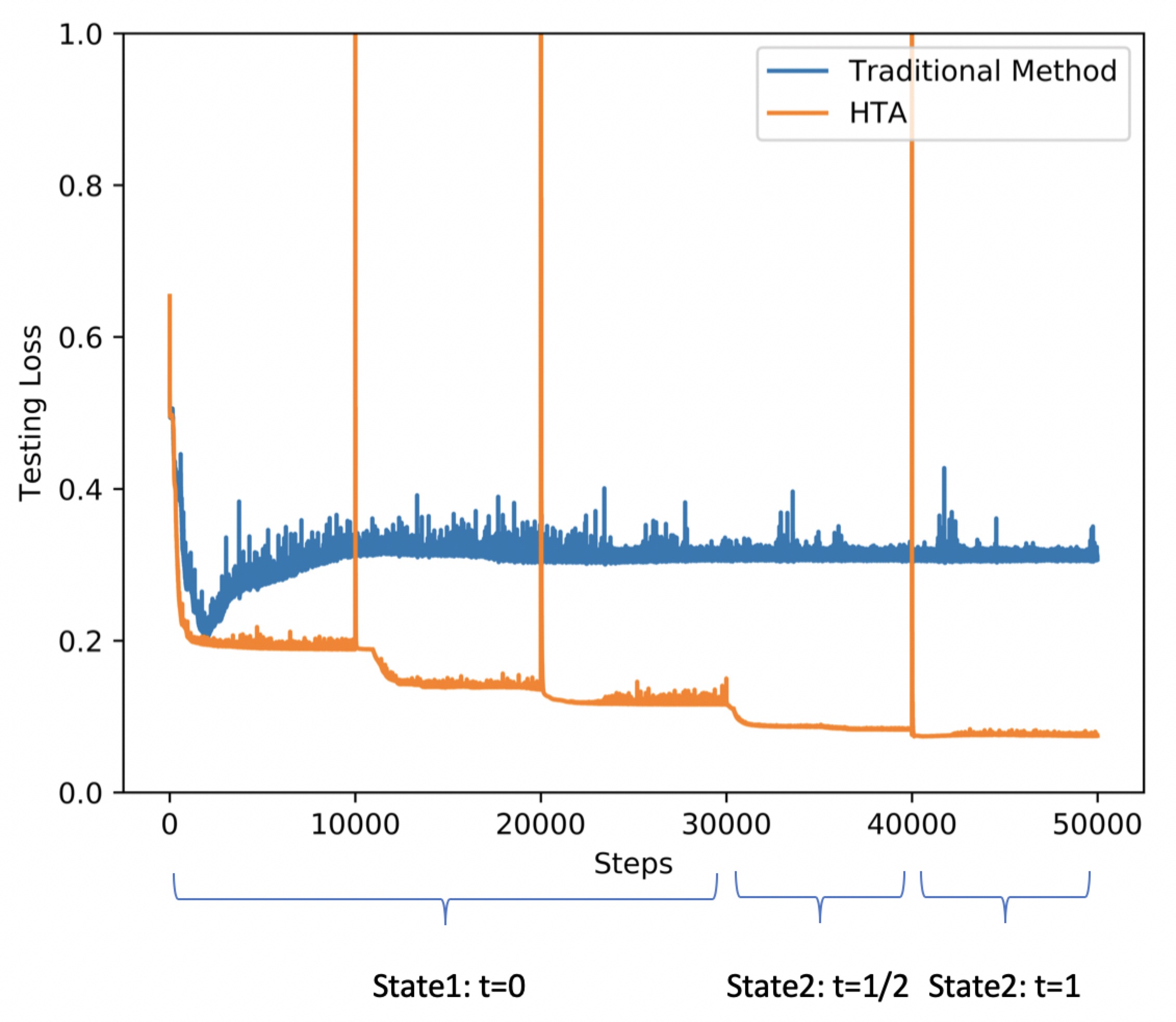}}
			\subfigure[Testing loss of  $n=6$]
			{\includegraphics[width=3.5in, angle=0]{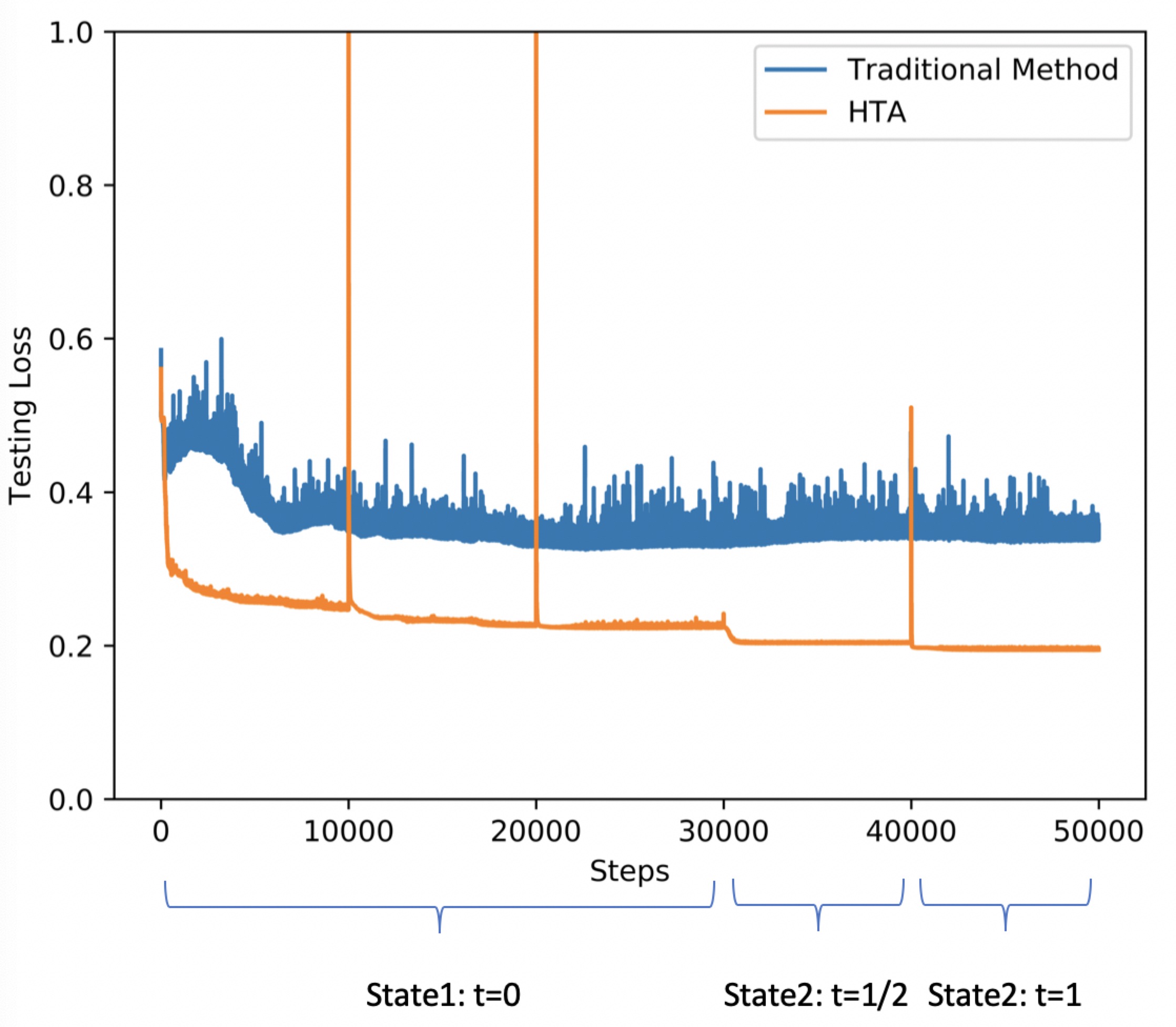}}
			\caption{Testing loss of approximating $\sin(x_1 + x_2 + \cdots + x_n)$.}
			\label{sin_testing_loss_hd}
		\end{figure*}
		\begin{figure}[h]
			\centering
			\includegraphics[width=0.7\textwidth]{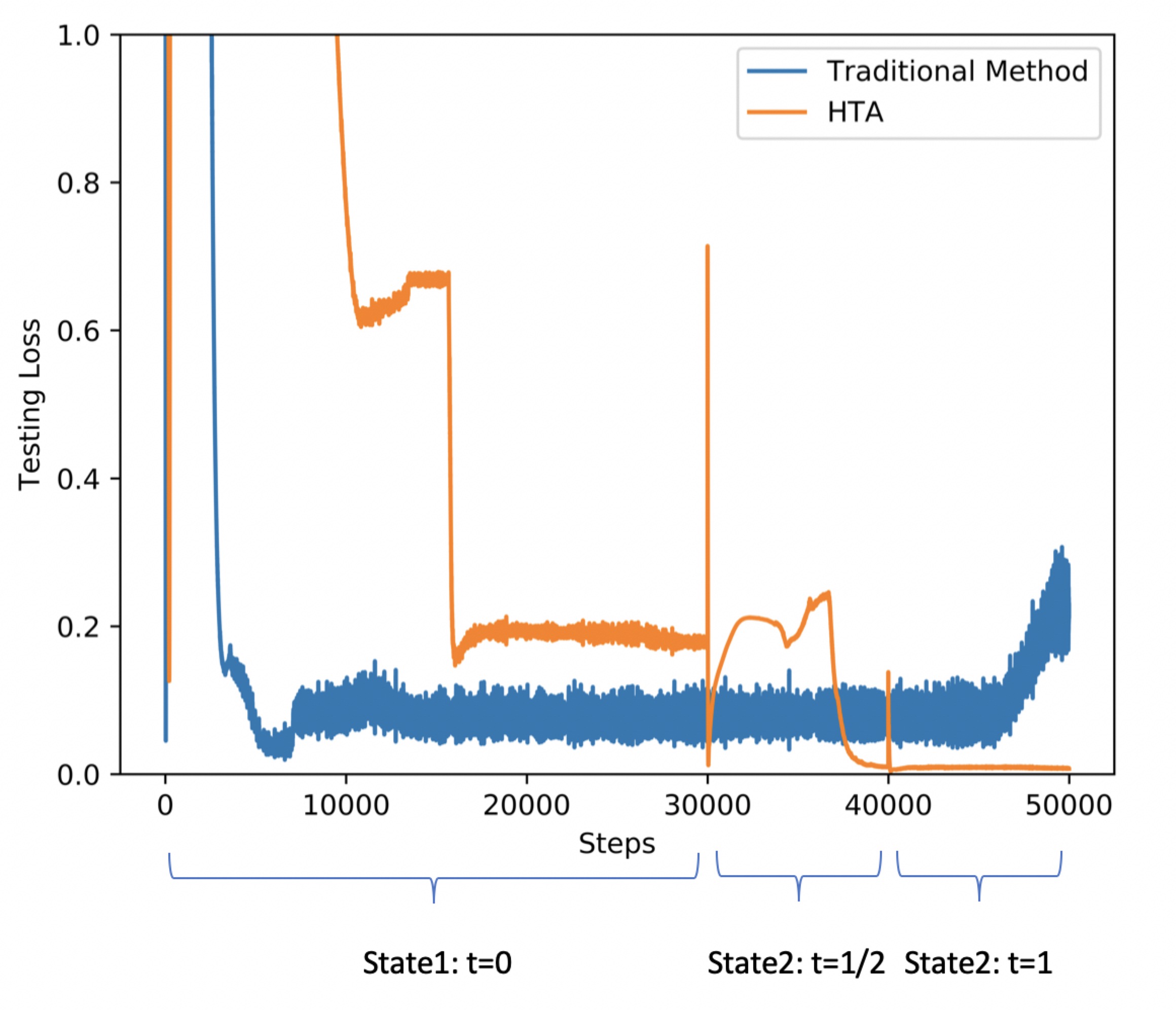}
			\caption{Testing loss of the Van der Pol equation.}\label{vdp}
		\end{figure}
		\begin{figure}[!t]
			\centering
			\subfigure[The traditional method]
			{\includegraphics[width=3.5in, angle=0]{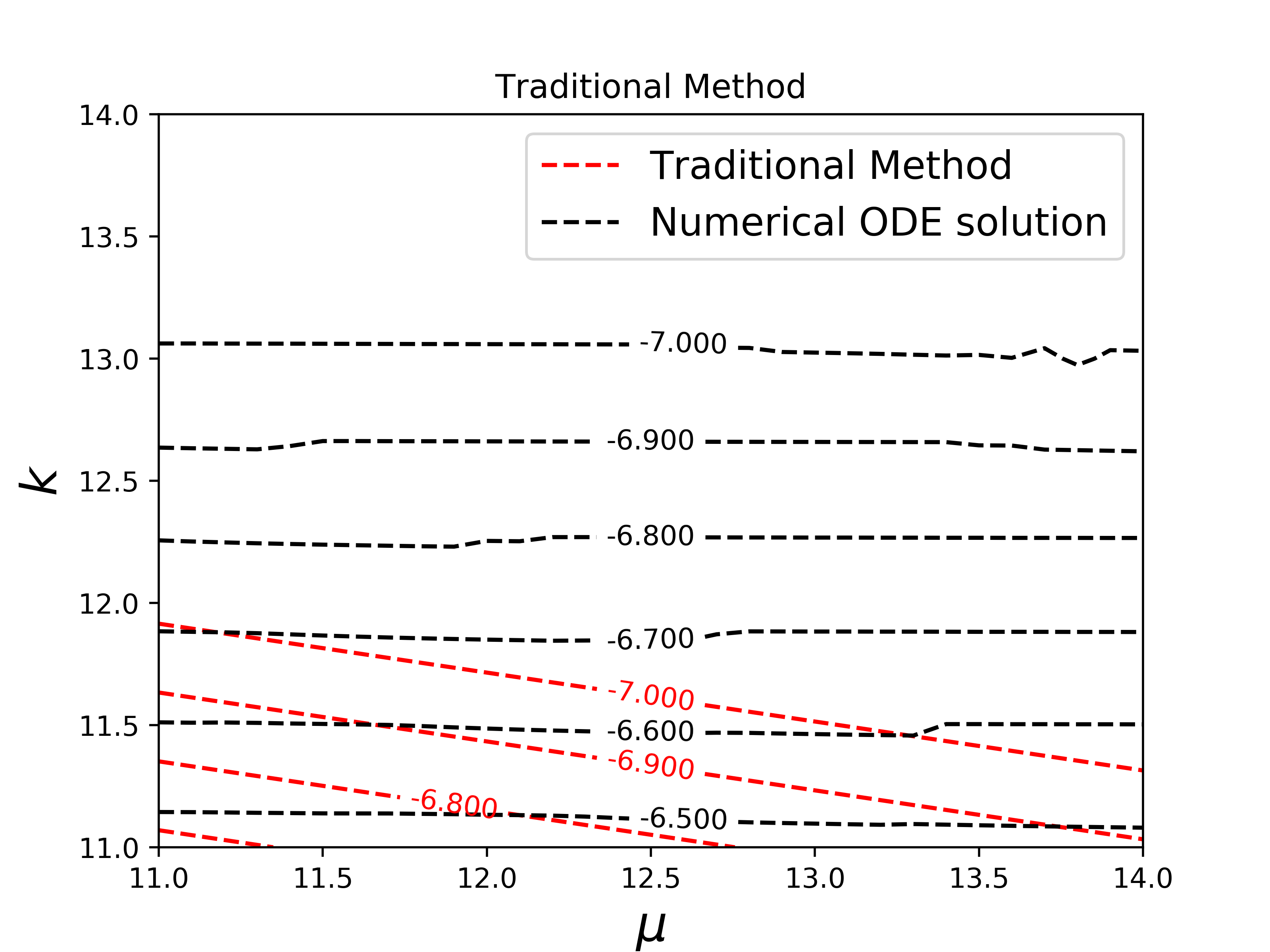}}
			\subfigure[The HTA]
			{\includegraphics[width=3.5in, angle=0]{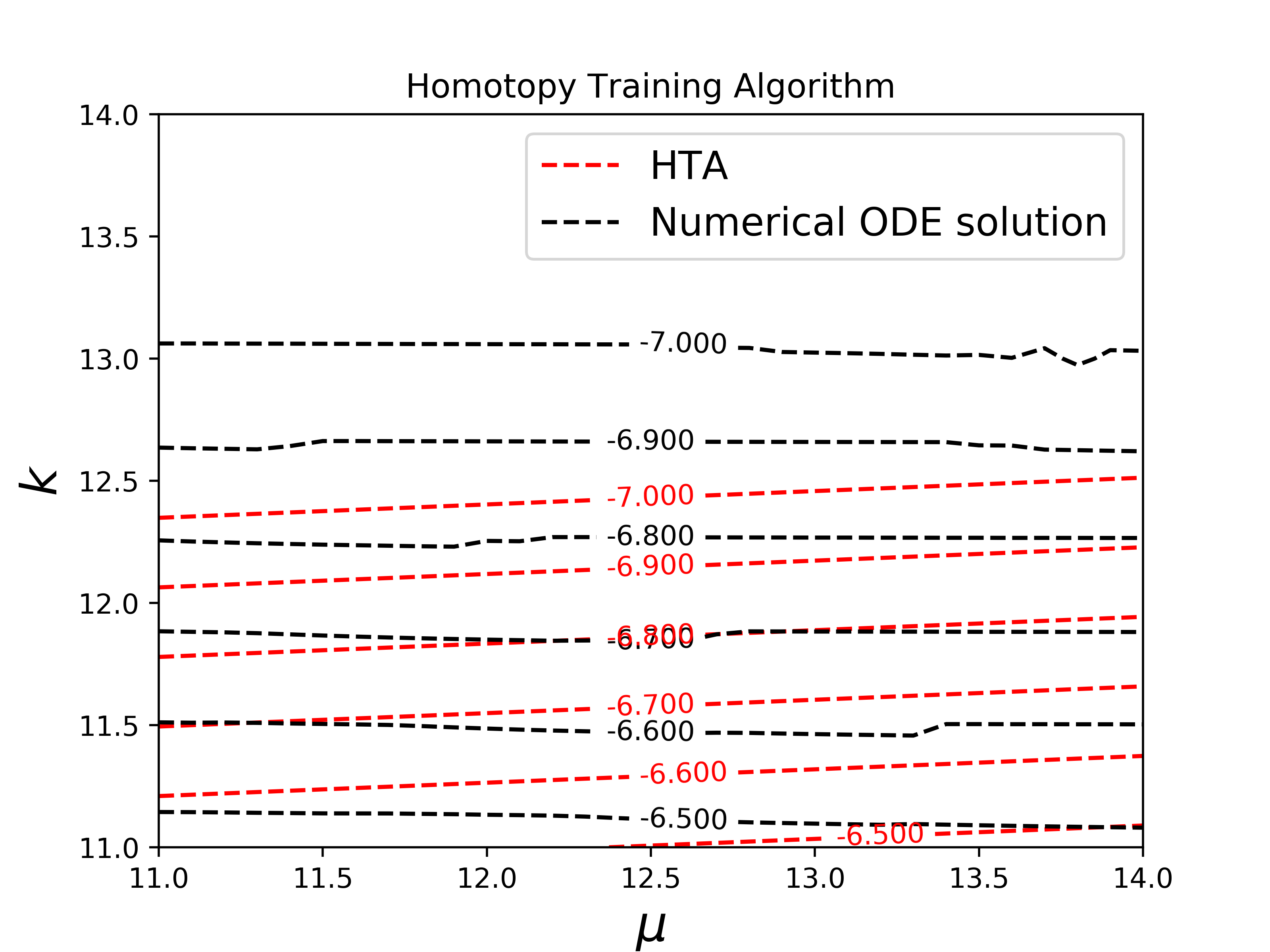}}
			\caption{Comparisons between the HTA and the traditional method by contour plots of $y(1;\mu,k)$.}
			\label{vdp_equation}
		\end{figure}

	\begin{figure}[h]
		\centering
		\includegraphics[width=\textwidth]{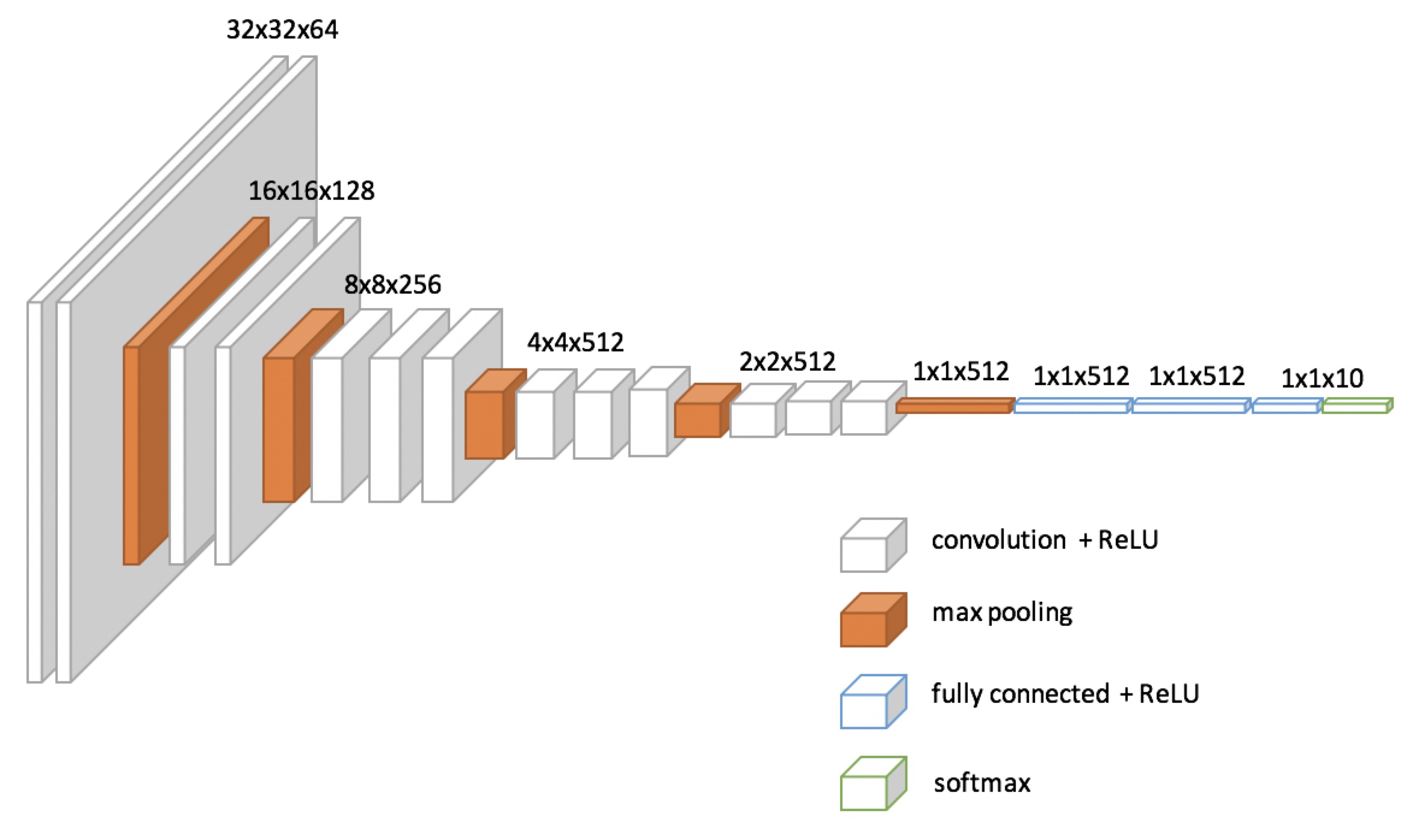}
		\caption{The structure of VGG models that consist of a convolution \& pooling neural network and
		a fully connected neural network. The HTA is applied on the fully connected layers only.}\label{VGG}
	\end{figure}

		\begin{figure}[h]
			\centering
			\includegraphics[width=0.6\textwidth]{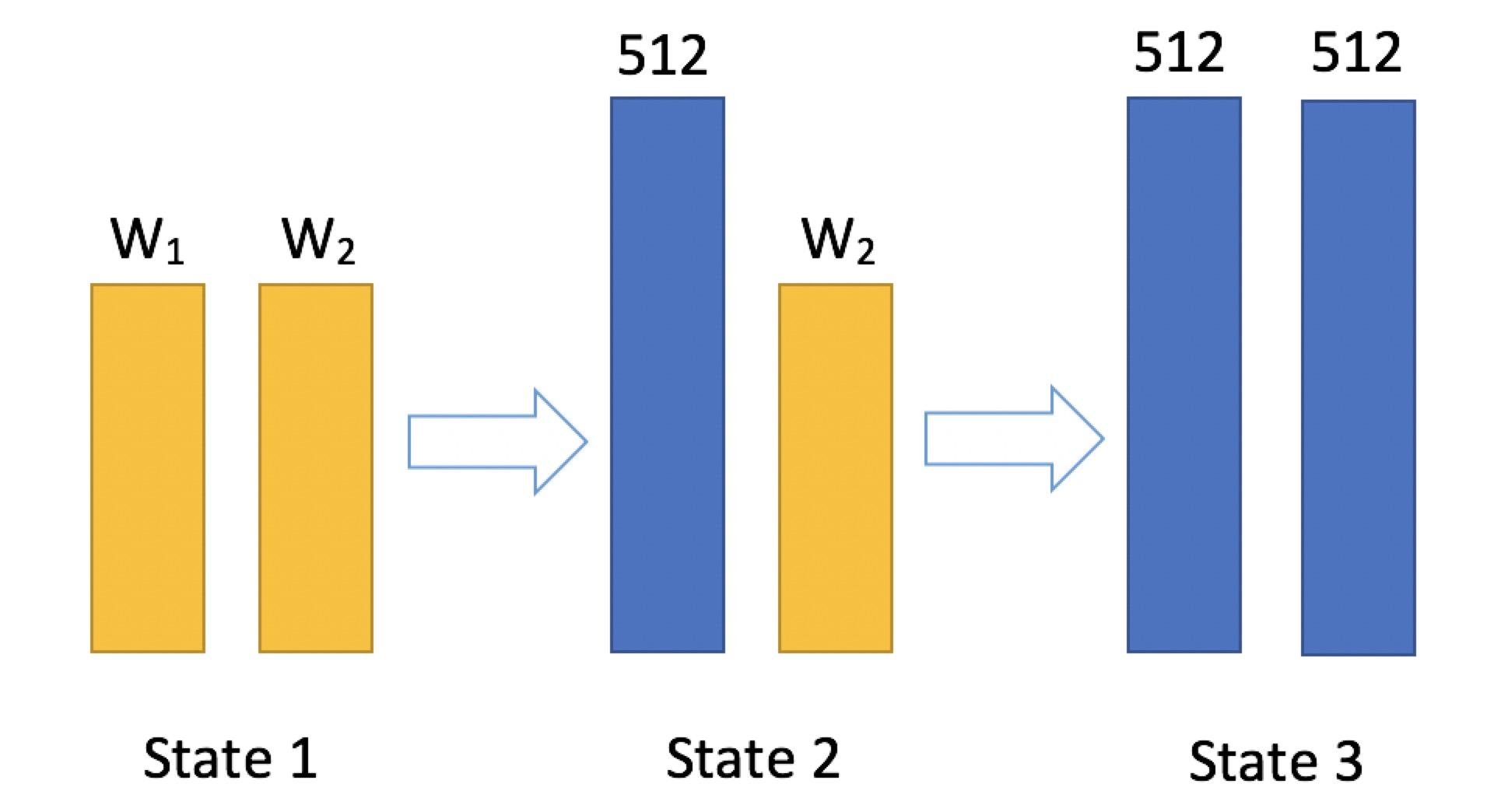}
			\caption{Three states of the VGG with the HTA.}\label{3states}
		\end{figure}

		\begin{figure}[h]
			\centering
			\includegraphics[width=0.7\textwidth]{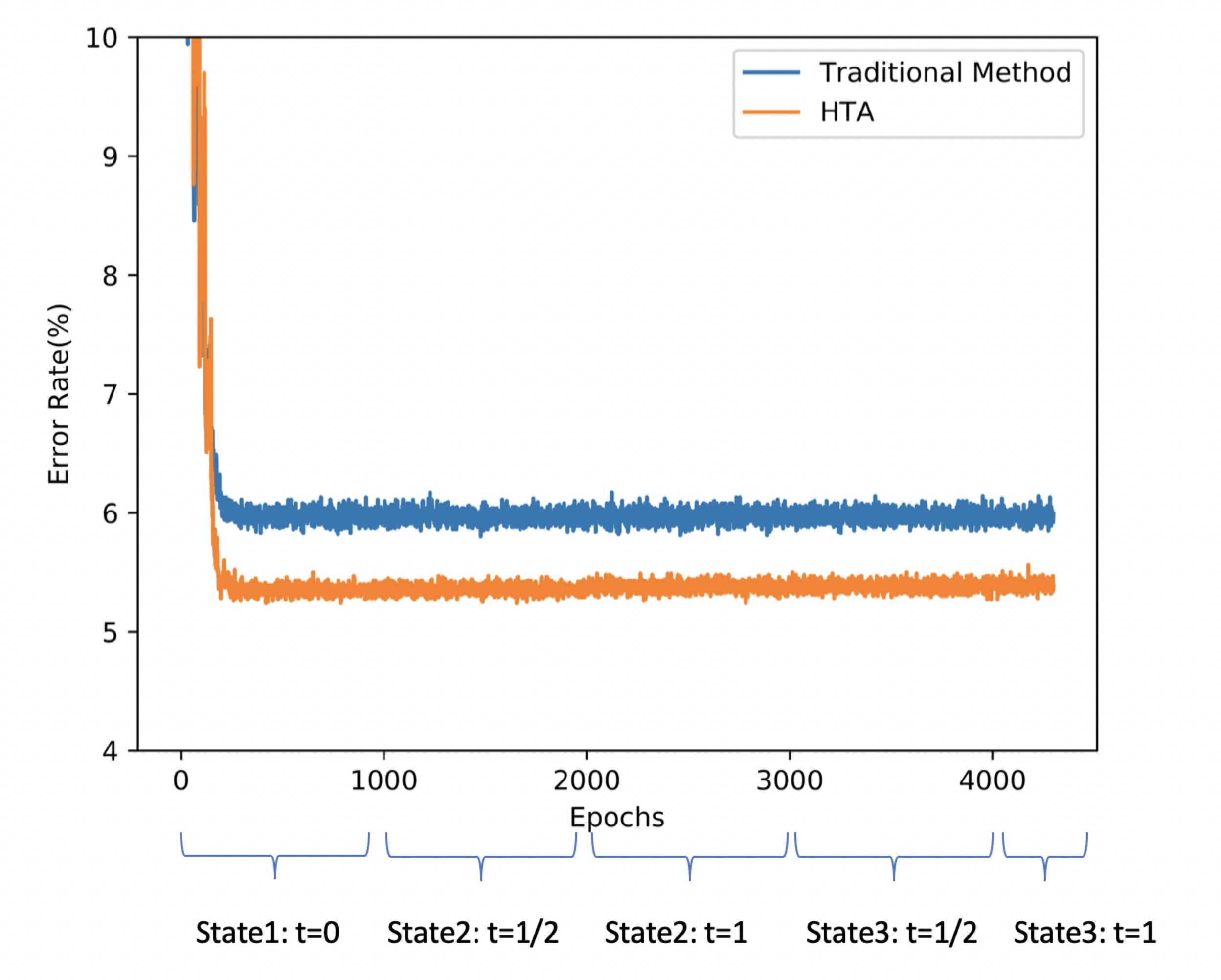}
			\caption{Comparisons of error rate for VGG13 between the HTA and the traditional method. }\label{vgg13_bn}
		\end{figure}

\noindent The output for table is:

		\begin{table}
			\begin{center}
				\begin{tabular}{| c | c | c | c |}
					\hline
					Dimensions (n) & Test Loss & Test Loss with HTA & Number of grid points \\ \hline
					1 & $0.078$ & $0.005$  & $10^2$ (uniform grid)\\ \hline
					2 & $0.195$ & $0.152$ & $10^4$ (uniform grid)\\ \hline
					3 & $0.347$ & $0.213$ & $10^6$ (uniform grid)\\ \hline
					4 & $0.393$ & $0.299$ & 2300 (sparse grid)\\ \hline
					5 & $0.493$ & $0.352$ & 5503 (sparse grid)\\ \hline
				\end{tabular}	
			\end{center}
			\caption{Testing loss of one-hidden-layer NNs.}
			\label{tab_testing_loss}	
		\end{table}

		\begin{table}
			\begin{center}
				\begin{tabular}{| c | c | c | c |}
					\hline
					Dimensions (n) & Test Loss & Test Loss with HTA & Number of gird points \\ \hline
					5 & $0.307$ & $0.074$ & 5503 (sparse grid)\\ \hline
					6 & $0.340$ & $0.194$ & 10625 (sparse grid)\\ \hline
					7 & $0.105$ & $0.029$ & 18943 (sparse grid)\\ \hline
					8 & $0.079$ & $0.022$ & 31745 (sparse grid)\\ \hline
				\end{tabular}	
			\end{center}
			\caption{Testing loss of two-hidden-layer NNs.}			\label{tab_testing_loss_hd}	
		\end{table}

	\begin{table}
			\begin{center}
				\begin{tabular}{| c | c | c |}
					\hline
					sample points $(\mu_i,k_i)$ & optima of traditional method $(\mu^*_i,k^*_i)$ &  optima of HTA $(\mu^*_i,k^*_i)$ \\ \hline
					$(11.1,12.9)$ & $(11.9, 11.6)$ & $(11.9, 12.6)$ \\ \hline
					$(11.9, 13.1)$ & $(11.9, 11.7)$ & $(11.9, 12.7)$ \\ \hline
					$(12.6, 11.4)$ & $(11.0, 11.0)$ & $(12.0, 11.4)$ \\ \hline
					$(13.2, 12.8)$ & $(11.9, 11.5)$ & $(11.9, 12.5)$ \\ \hline
					$(13.9, 11.1)$ & $(11.0, 11.0)$ & $(12.0, 11.2)$ \\ \hline
				\end{tabular}
			\end{center}
			\caption{Parameter estimation results.}			\label{para}
		\end{table}

		\begin{table}
			\begin{center}
				\begin{tabular}{| p{2.5cm} | p{2.5cm} | p{2.5cm} | p{2.5cm} |}
					\hline
					Base Model Name & Original Error Rate & Error Rate with HTA & Rate of Imrpovement(ROIs)\\ \hline
					VGG11 & $7.83\%$ & $7.02\%$ & $10.34\%$ \\ \hline
					VGG13 & $5.82\%$ & $5.14\%$ & $11.68\%$ \\ \hline
					VGG16 & $6.14\%$ & $5.71\%$ & $7.00\%$ \\ \hline
					VGG19 & $6.35\%$ & $5.88\%$ & $7.40\%$ \\ \hline
				\end{tabular}
			\end{center}
			\caption{The comparison between the HTA and the traditional method of VGG models on CIFAR-10}			\label{vgg_result}
		\end{table}

		\begin{table}
			\begin{center}
				\begin{tabular}{| c| c | c | c | c |}
					\hline
					Base Model Name & Original Error Rate & Error Rate with OSF & $w_1$ & $w_2$ \\ \hline
					VGG11 & $7.83\%$ & $7.37\%$ & 480 & 20 \\ \hline
					VGG13 & $5.82\%$ & $5.67\%$ & 980 & 310 \\ \hline
					VGG16 & $6.14\%$ & $5.91\%$ & 752 & 310 \\ \hline
					VGG19 & $6.35\%$ & $6.05\%$ & 752 & 310 \\ \hline
				\end{tabular}
			\end{center}
			\caption{The results of the algorithm for finding the optimal structure when applied to CIFAR-10: $w_i$ stands for the optimal width of $i$-th hidden layer.}			\label{BSF}
		\end{table}

		\begin{algorithm}
		\caption{The HTA algorithm for a neural network with two hidden layers.} 
		\begin{algorithmic}[1]
			\STATE{Solve the optimization problem $\mathcal{L}(H_1(x;\theta,0),y$) and denote the solution as $\theta_1^*$;}
			\STATE{Set an initial guess as $\theta^0=\theta_1^*\cup\theta_2^0$ and $N=\frac{1}{\delta t}$ where $\delta t$ is the homotopy stepsize;}
			\FOR {$i = 1,\cdots, N$}
			\STATE{Solve $\theta^i=argmin\mathcal{L}(H_1(x;\theta,i\delta t),y)$ by using $\theta^{i-1}$ as the initial guess}
				\ENDFOR
		\end{algorithmic}\label{alg}
		\end{algorithm}


\bibliographystyle{plain}

\bibliography{references}

\end{document}